\documentclass[a4paper]{article}

\usepackage{graphicx}
\usepackage{amsmath, amssymb, amsfonts, amsthm}
\usepackage{hyperref}
\usepackage{xspace}
\usepackage{color}
\usepackage{marginnote}
\usepackage{fancybox}
\usepackage{tikz-cd}

\usepackage{pgfplots}
\pgfplotsset{compat=newest}
\usepgfplotslibrary{fillbetween}

\def\R{{\mathbb R}}

\def\al{\alpha}
\def\be{\beta}
\def\la{\lambda}

\DeclareMathOperator{\e}{e}
\DeclareMathOperator{\im}{im}
\DeclareMathOperator{\sign}{sign}
\DeclareMathOperator{\rank}{rank}
\DeclareMathOperator{\diag}{diag}
\DeclareMathOperator{\lineal}{lineal}

\newcommand{\trans}{\mathsf{T}} 
\newcommand{\id}{\mathrm{I}}
\newcommand{\bd}{\partial}
\newcommand{\dd}[2]{\frac{\mathrm{d} #1}{\mathrm{d} #2}}
\newcommand{\pd}{\mathcal{D}_>}

\newcommand{\ce}[1]{{\sf{#1}}}

\newcommand{\had}{\circ} 
\newcommand{\hads}{\odot} 
\newcommand{\xa}{x^*} 

\newcommand{\EE}{\mathcal{E}}
\newcommand{\Ac}{\mathcal{A}_{k,\EE}}
\newcommand{\Acc}{\mathcal{A}'_{k,\EE}}
\newcommand{\ST}{\mathcal{\tilde S}_{k,\EE}}
\newcommand{\CE}{\mathcal{\tilde C}_{\EE}}
\newcommand{\JJ}{\mathcal{J}}

\newcommand{\x}{\circ}
\newcommand{\xx}{{23}}
\newcommand{\xxx}{{12}}
\newcommand{\xxxx}{{34}}

\newtheorem{thm}{Theorem}
\newtheorem{pro}[thm]{Proposition}
\newtheorem{lem}[thm]{Lemma}
\newtheorem{cor}[thm]{Corollary}
\newtheorem{fac}[thm]{Fact}
\theoremstyle{definition}
\newtheorem{dfn}[thm]{Definition}
\newtheorem{exa}[thm]{Example}
\newtheorem{rem}[thm]{Remark}

\makeatletter
\def\blfootnote{\xdef\@thefnmark{}\@footnotetext}
\makeatother

\newcommand{\red}[1]{{\color{red}#1}}
\newcommand{\trueblue}[1]{{\color{blue}#1}}
\newcommand{\green}[1]{{\color{green}#1}}

\newcommand{\blue}[1]{{#1}}

\begin{document}

\title{%
Sufficient conditions for linear stability \\
of complex-balanced equilibria \\
in generalized mass-action systems
}

\author{
Stefan~M\"uller, Georg~Regensburger
}


\maketitle

\begin{abstract}
\noindent
Generalized mass-action systems are power-law dynamical systems arising from chemical reaction networks.
Essentially, every nonnegative ODE model used in chemistry and biology (for example, in ecology and epidemiology) 
and even in economics and engineering can be written in this form. 
Previous results have focused on existence and uniqueness of special steady states (complex-balanced equilibria) for all rate constants,
thereby ruling out multiple (special) steady states.
Recently, necessary conditions for linear stability have been obtained. 

In this work, we provide sufficient conditions for the linear stability of complex-balanced equilibria for all rate constants
(and also for the non-existence of other steady states). 
In particular, via sign-vector conditions (on the stoichiometric coefficients and kinetic orders), 
we guarantee that the Jacobian matrix is a $P$-matrix. 
Technically, we use a new decomposition of the graph Laplacian which allows to consider orders of (generalized) monomials.
Alternatively, we use cycle decomposition which allows a linear parametrization of all Jacobian matrices.
In any case, we guarantee stability without explicit computation of steady states.
We illustrate our results in examples from chemistry and biology: 
generalized Lotka-Volterra systems and SIR models, a two-component signaling system, and an enzymatic futile cycle.

\smallskip \noindent
{\bf Keywords:}
chemical reaction networks, generalized mass-action kinetics, 
polynomial/power-law dynamical systems,
stability, $P$-matrix

\end{abstract}

\blfootnote{
\scriptsize

\noindent
{Stefan~M\"uller} \\
Faculty of Mathematics, University of Vienna, Oskar-Morgenstern-Platz 1, 1090 Wien, Austria \\[1ex]
{Georg Regensburger} \\
Institut f\"ur Mathematik, Universit\"at Kassel, Germany \\[1ex]
Corresponding author: 
\href{mailto:st.@univie.ac.at}{st.mueller@univie.ac.at}
}


\section{Introduction}


A chemical reaction network (CRN) describes interactions 
between a set of 
species based on a directed graph. 
Under the assumption of (generalized) mass-action kinetics (MAK/GMAK),
it leads to polynomial/power-law dynamical systems for the species concentrations.
Importantly, these (generalized) mass-action systems reach far beyond chemistry and biology.
In fact, every polynomial/power-law model used in ecology and epidemiology
and even in economy and engineering 
can be written in this form. 
The relative simplicity of these systems is appealing for dynamical modeling,
parameter estimation, qualitative analysis, and optimal control.
In applications, it is extremely useful to have strong results that depend mainly on the graph structure of the model,
but not on the parameter values. 

%
%

Indeed, mass-action systems that permit special steady states, so-called complex-balanced equilibria (CBE),
display remarkably robust dynamics.
If one positive equilibrium is complex-balanced, then so is every other equilibrium, 
and all \blue{positive} equilibria are asymptotically stable (as shown via the classical entropy-like Lyapunov function)~\cite[Theorem~6A]{HornJackson1972}.
Moreover,
there exists a unique positive equilibrium in every forward-invariant set (stoichiometric class)~\cite[Lemma~4B]{HornJackson1972}.
Finally, a mass-action system admits a CBE, {\em for all rate constants}, 
if and only if the components of the graph are strongly connected (the network is weakly reversible) 
and the deficiency $\delta=m-\ell-s$ is zero (with numbers of vertices~$m$, components~$\ell$, and independent species~$s$)~\cite[Theorem~4A]{Horn1972}.

However, MAK is an assumption that holds for elementary reactions in homogeneous and dilute solutions.
In intracellular environments, which are highly structured and crowded, and for reaction mechanisms,
more general kinetics are needed. 
A natural extension is GMAK (in the version of~\cite{MuellerRegensburger2014}), 
where the kinetic orders may differ from the stoichiometric coefficients,
leading to power-law dynamical systems. 
(For a reaction $1\ce{X}+1\ce{Y}\to\ce{Z}$, the rate is not $\sim \! x^1y^1$, but $\sim \! x^\al y^\be$ with arbitrary, not necessarily integer $\al,\be$.)
Besides the limited validity of the mass-action ``law'',
a CRN with MAK may not have zero deficiency and may not be weakly reversible,
but there may be a {\em dynamically equivalent}, ``translated'' CRN with GMAK that has the desired properties~\cite{Johnston2014,Johnston2015,TonelloJohnston2018,JohnstonMuellerPantea2019}.

%

For the resulting generalized mass-action systems, existence and uniqueness of CBE
 (in every stoichiometric class and for all rate constants) are well understood~\cite{MuellerRegensburger2012,MuellerRegensburger2014,Mueller2016,MuellerHofbauerRegensburger2019,CraciunMuellerPanteaYu2019},
in particular, generalizations of the deficiency zero theorem have been provided,
see also the discussion in Section~\ref{sec:previous}.
However, much less is known about the stability of CBE.
Even planar S-systems with a unique CBE
display rich dynamical behavior, including super/sub-critical or degenerate Hopf bifurcations,
centers, and up to three limit cycles, see~\cite{BorosHofbauerMueller2017,BorosHofbauerMuellerRegensburger2018,BorosHofbauerMuellerRegensburger2019,BorosHofbauer2019}.


For certain classes of networks and monotonically increasing kinetics,
``robust'' (often piecewise linear) Lyapunov functions have been constructed~\cite{BlanchiniGiordano2014,AlRadhawiAngeli2016,AlRadhawi2020}
that guarantee ``structural attractivity'' (unique globally asymptotically stable steady states).
However, no Lyapunov functions are known for CBE of arbitrary generalized mass-action systems.
Recently, the stability problem has been approached by linearization,
and necessary conditions for linear stability (for all rate constants) have been obtained~\cite{BorosMuellerRegensburger2020}.

%

In this work, we provide sufficient conditions for the linear stability of CBE for all rate constants
and also for the preclusion of other steady states.
In the case of mass-action systems, 
\blue{asymptotic stability and absence of non-complex-balanced equilibria were shown via the same method (a Lyapunov function),
whereas linear stability was proven using different methods~\cite{SiegelJohnston2008,Feinberg2019,BorosMuellerRegensburger2020}.}
In the case of generalized mass-action systems, 
we show \blue{linear stability of CBE and preclusion of other steady states} via a new decomposition of the graph Laplacian 
and orders of the generalized monomials.
The foundation of our stability results is a theorem by Carlson~\cite{Carlson1974},
which states that a sign-symmetric $P$-matrix is stable.
The theorem holds for matrices with full rank, and we extend it to stability on a linear subspace.

Ultimately, we obtain sign-vector conditions on the stoichiometric coefficients and kinetic orders
(as in our existence and uniqueness results)
to guarantee that the Jacobian matrix of a CBE is a $P$-matrix.
On the one hand,
we do not strictly extend the classical results
since, even in mass-action systems, the Jacobian matrix of a CBE need not be a $P$-matrix, cf.~\cite{Banaji2007}.
On the other hand,
our main results do not just hold for a given generalized mass-action system,
but for all systems (with the same stoichiometric and kinetic-order data)
based on a graph with the same components, but arbitrary edge set.


Finally,
for a reader interested in modeling,
we provide two examples (Lotka-Volterra systems and SIR models)
in order to demonstrate the expressivity of generalized mass-action systems.
The examples are small, but fundamental models of ecological and epidemiological processes.
They allow to illustrate our results most of which hold for systems of arbitrary dimension.
A reader interested in the technical exposition may skip the examples
and continue with the subsections `\nameref{sec:org}' and `\nameref{sec:not}' at the end of this introduction.

{\em Example I.} The Lotka-Volterra predator-prey system
\[
\dd{x}{t} = b \cdot x - c \cdot x y
\quad \text{and} \quad
\dd{y}{t} = +c \cdot x y - d \cdot y
\]
with abundance of prey and predator, $x$ and $y$, respectively,
birth rate (of prey)~$b$, death rate (of predator)~$d$, and interaction rate~$c$
can be written as a chemical reaction network assuming MAK:
$\ce{X\to2X}$, $\ce{X+Y\to2Y}$, and $\ce{Y\to 0}$. 
Note that the first reaction has the same net effect as $\ce{0\to X}$,
but with kinetics $\sim \! x^1$ (which we encode as $1\,\ce{X}$).
Further, the second reaction has the same net effect as $\ce{X\to Y}$,
but with kinetics $\sim \! x^1 y^1$ (which we encode as $1\,\ce{X}+1\,\ce{Y}$). 
Using GMAK, we can write the system as
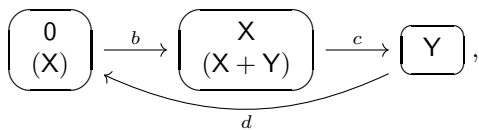
\[
\begin{tikzcd}[ampersand replacement=\&]
\mbox{\ovalbox{$\begin{array}{c} \ce{0} \\ (\ce{X}) \end{array}$}} \arrow[r,"b"] \& \mbox{\ovalbox{$\begin{array}{c} \ce{X} \\ (\ce{X}+\ce{Y}) \end{array}$}} \arrow[r,xshift=+0ex,"c"] \&
\mbox{\ovalbox{$\begin{array}{c} \ce{Y} 
\end{array}$}} \; , \arrow[ll,"d",bend left=25]
\end{tikzcd}
\]
where we put the kinetic orders in brackets 
(if they differ from the stoichiometric coefficients).

Whereas the MAK system has deficiency $\delta=6-3-2=1$ (6~vertices in 3~connected components and 2~independent variables)
and is not weakly reversible (the network components are not strongly connected),
the GMAK system has deficiency $\delta=3-1-2=0$ and is weakly reversible.
Of course, the classical Lotka-Volterra system is not stable.
In Example~\ref{exa:Lotka}, we consider more general kinetics,
namely $\sim \! x^\al$ for $\ce{X\to2X}$ and $\sim \! x y^\be$ for $\ce{X+Y\to2Y}$ 
(for the birth and interaction processes).
First, deficiency zero and weak reversibility allow to apply (generalizations of) classical existence and uniqueness results.
Second, for $\al,\be<1$, results from this work guarantee the linear stability of the unique positive steady state 
for all rate constants $b,c,d$.

{\em Example II.} The SIR epidemiology model (for constant population size)
\[
\dd{x}{t} = b - c \cdot x y - d \cdot x
\quad \text{and} \quad
\dd{y}{t} = c \cdot x y - r \cdot y - d \cdot y
\]
with fractions of susceptible and infected individuals, $x$ and $y$, respectively, 
birth and death rates~$b=d$, infection rate~$c$, and recovery rate~$r$
can also be written as a chemical reaction network assuming MAK:
$\ce{0\to X}$, $\ce{X+Y\to2Y}$, $\ce{X\to 0}$, and $\ce{Y\to 0}$. 
Again, we can rewrite the system using GMAK,
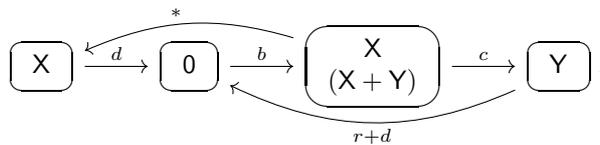
\[
\begin{tikzcd}[ampersand replacement=\&]
\mbox{\ovalbox{$\begin{array}{c} \ce{X} 
\end{array}$}} \arrow[r,"d"] \&
\mbox{\ovalbox{$\begin{array}{c} \ce{0} 
\end{array}$}} \arrow[r,"b"] \& \mbox{\ovalbox{$\begin{array}{c} \ce{X} \\ (\ce{X}+\ce{Y}) \end{array}$}} \arrow[r,xshift=+0ex,"c"] \arrow[ll,"*"', bend right=20] \&
\mbox{\ovalbox{$\begin{array}{c} \ce{Y} 
\end{array}$}} \; , \arrow[ll,"r+d",bend left=25]
\end{tikzcd}
\]
where the ``reaction'' $\ce{X \overset{*}{\to} X}$ does not contribute to the dynamics, but makes the network weakly reversible.
Further, the network has effective deficiency $\delta'=3-1-2=0$
(4~vertices, but only 3~different stoichiometric data ($\ce{0,X,Y}$) in 1~connected component and 2~independent variables),
cf.~\cite{JohnstonMuellerPantea2019}.
An explicit analysis shows that,
only for $R_0 = c / (r+d) > 1$, there exists a positive steady state (the epidemic equilibrium), which is linearly stable.
Again, we consider more general kinetics,
in particular, $\sim \! x y^\be$ for $\ce{X+Y\to2Y}$ (the infection process).
For $\be<1$, existence and uniqueness results apply,
and results from this work guarantee the linear stability of the unique positive steady state (the epidemic equilibrium) for all rate constants $b,c,d,r$.

Obviously, the qualitative dynamics of the classical SIR model is not robust
with respect to small perturbations of the ``infection law'' (from $\sim \! x y$ to $\sim \! x y^\be$).
An epidemiological interpretation of this fact is beyond the scope of this work.


\paragraph*{Organization of the work.} \label{sec:org}

In Section~\ref{sec:gmas}, we introduce generalized mass-action-systems.
In particular, in \ref{sec:previous}, we list previous results on existence, uniqueness, and stability of CBE
and define relevant notions of stability (on a linear subspace).
In~\ref{sec:new}, we introduce a new decomposition of the graph Laplacian,
and in~\ref{sec:rewrite} we rewrite the Jacobian matrix using the new decomposition and cycle decomposition.
In~\ref{sec:orders}, we consider monomial evaluation orders and hyperplane arrangements.
Regions with given monomial order are faces of a hyperplane arrangement and can be represented by sign-vectors of a linear subspace.

In Section~\ref{sec:main}, we present our main results.
As a first step, in~\ref{sec:noother}, 
we use the new decomposition of the graph Laplacian and monomial evaluation orders
to rule out steady states other than CBE.
In~\ref{sec:stab}, we use the same methods to guarantee linear stability of CBE
in case of a full dimensional stoichiometric subspace.
Finally, in~\ref{sec:linear}, we consider the general case.
Here, we use two alternative strategies: 
either a transformation of the dynamical system to independent variables  
or a generalization of standard stability results to matrices without full rank.
For all cases, we provide examples from the biochemical literature.
Finally, in~\ref{sec:sign}, we discuss relations between the sign vector conditions in our main results.

In the \nameref{app}, 
we provide basic results regarding 
order theory and oriented matroids.


\paragraph*{Notation.} \label{sec:not}

We denote the positive real numbers by $\R_>$ and the nonnegative real numbers by $\R_\ge$. 
We write $x>0$ for $x \in \R^n_>$ and $x \ge 0$ for $x \in \R^n_\ge$.
For vectors $x,y \in \R^n$, we denote their scalar product by $x \cdot y$ and their componentwise (Hadamard) product by $x \had y$.
For $x \in \R^n_>$ and $y \in \R^n$, we define the monomial $x^y = \prod_{i=1}^n x_i^{y_i} \in \R_>$.

For a vector $x \in \R^n$,
we obtain the sign vector $\sign(x) \in \{-,0,+\}^n$ by applying the sign function componentwise.
For a subset $S \subseteq \R^n$,
we write
\[
\sign(S) = \{ \sign(x) \mid x \in S \} \subseteq \{-,0,+\}^n .
\]
Conversely, for a sign vector $\sigma \in \{-,0,+\}^n$,
we write
$
\sign^{-1}(\sigma) = \{ x \in \R^n \mid \sign(x) \in \sigma \} 
$
for the corresponding open (in general, not full-dimensional) orthant.
Finally, we write
$
\Sigma(S) = \sign^{-1}(\sign(S)) 
$
for the union of all orthants that $S$ intersects.
Clearly,
$\Sigma(S) = 
\bigcup_{D \in \pd} D(S)$,
where $\pd \subset \R^{n \times n}_\ge$ is the set of diagonal matrices with positive diagonal.




The product on $\{-,0,+\}$ is defined in the obvious way.
For $\sigma,\tau \in \{-,0,+\}^n$, 
we define the componentwise product  $\sigma \hads \tau \in \{-,0,+\}^n$ as $(\sigma \hads \tau)_i = \sigma_i \tau_i$.
The inequalities $0<-$ and $0<+$ induce a partial order on $\{-,0,+\}^n$:
for sign vectors $\sigma, \tau \in \{-,0,+\}^n$, we write $\sigma \le \tau$ if the inequality holds componentwise.
In particular, $\sigma \hads \tau \ge 0$ if and only if $\sigma_i \tau_i \ge 0$ for all $i$
($\sigma$ and $\tau$ are ``harmonious'').
Correspondingly, we write $\sigma \hads \tau > 0$ if $\sigma \hads \tau \ge 0$ and $\sigma \hads \tau \neq 0$
($\sigma$ and $\tau$ are ``harmonious'' and ``overlapping'').
Finally, for $\mathcal{T} \subseteq \{-,0,+\}^n$, we introduce the lower closure
$
\mathcal{T}^\downarrow = 
\{ \sigma 
\mid \sigma \le \tau \text{ for some } \tau \in \mathcal{T} \} 
$
and the total closure
\[
\mathcal{T}^\updownarrow = 
\{ \sigma 
\mid \sigma \le \tau \text{ or } \sigma \ge \tau \text{ for some nonzero } \tau \in \mathcal{T} \} .
\]
For fundamental results on sign vectors of linear subspaces,
see Appendix~\ref{app:sign}.




\section{Generalized mass-action systems} \label{sec:gmas}

A {\em generalized chemical reaction network} $(G,y,\tilde y)$ is given by a simple directed graph $G=(V,E)$
and two maps $y \colon V \to \R^n$ and $\tilde y \colon V_s \to \R^n$.
Every vertex $i \in V$
is labeled with a {\em (stoichiometric) complex} $y(i) \in \R^n$, 
and every source vertex $i \in V_s \subseteq V$ is labeled with a {\em kinetic-order complex} $\tilde y(i) \in \R^n$.
\blue{In particular,} $G=(V,E)$ has $\ell$ connected components $G_\la = (V_\la, E_\la)$, $\la=1,\ldots,\ell$.
If every component of $G$ is strongly connected, we call $G$ and $(G,y,\tilde y)$ {\em weakly reversible}.

A {\em generalized mass-action system} $(G_k,y,\tilde y)$ is a generalized chemical reaction network $(G,y,\tilde y)$
together with edge labels $k \in \R^E_>$, resulting in the labeled digraph $G_k$.
Every edge $(i \to i') \in E$, representing the {\em chemical reaction} ${y(i) \to y(i')}$, is labeled with a {\em rate constant} $k_{i \to i'} > 0$.
If $G$ is weakly reversible, we call $G_k$ and $(G_k,y,\tilde y)$ weakly reversible.

The ODE system for the {\em species concentrations} $x \in \R^n_>$,
associated with the generalized mass-action system $(G_k,y,\tilde y)$, is given by
\begin{equation} \label{ode1}
\dd{x}{t} = \sum_{(i \to i') \in E} k_{i \to i'} \, x^{\tilde y(i)} \big( y(i')-y(i) \big) .
\end{equation}
The sum ranges over all reactions, 
and every summand is a product of the reaction rate $k_{i \to i'} \, x^{\tilde y(i)}$, involving a monomial determined by the kinetic-order complex of the educt,
and the difference $y(i')-y(i)$ of the stoichiometric complexes of product and educt.

Let $I_E, I_E^s \in \R^{V \times E}$ be the incidence and source matrices
of the digraph~$G$, 
and $A_k = I_E \diag(k) (I_E^s)^\trans \in \R^{V \times V}$ be the Laplacian matrix of the labeled digraph $G_k$.
(Explicitly, $(I_E)_{i,j \to j'} = -1$ if $i=j$, $(I_E)_{i,j \to j'} = 1$ if $i=j'$, and $(I_E)_{i,j \to j'} = 0$ otherwise.
Further, $(I_{E,s})_{i,j \to j'} = 1$ if $i=j$ and $(I_{E,s})_{i,j \to j'} = 0$ otherwise.
Finally, $(A_k)_{i,j} = k_{j \to i}$ if $(j \to i) \in E$, $(A_k)_{i,i} = - \sum_{(i \to j)\in E} k_{i \to j}$, and $(A_k)_{i,j} = 0$ otherwise.)

Now, 
the right-hand-side of the ODE system~\eqref{ode1} can be decomposed into stoichiometric, graphical, and kinetic-order contributions,
\begin{equation} \label{ode2}
\dd{x}{t} = Y I_E \diag(k) (I_{E,s})^\trans \, x^{\tilde Y} = Y A_k \, x^{\tilde Y} ,
\end{equation}
where $Y,\, \tilde Y \in \R^{n \times V}$ are the matrices of stoichiometric and kinetic-order complexes,
and $x^{\tilde Y} \in \R^V_>$ denotes the vector of monomials,
that is, $(x^{\tilde Y})_i = x^{\tilde y(i)}$.
(Note that columns of $\tilde Y$ corresponding to non-source vertices can be chosen arbitrarily,
since they do not specify the rate of any reaction.)

Clearly, the change over time lies in the {\em stoichiometric subspace}
\begin{equation}
S = \im (Y I_E) ,
\end{equation}
that is, $\dd{x}{t} \in S$.
Equivalently, trajectories are confined to cosets of $S$,
that is, $x(t) \in x(0)+S$.
For $x' \in \R^n_>$, the set $(x'+S) \cap \R^n_>$ is called a {\em stoichiometric class}.
The (stoichiometric) {\em deficiency} is given by
\begin{equation}
\delta 
= \dim(\ker Y \cap \im I_E) 
= m - \ell - \dim(S) ,
\end{equation}
where $m=|V|$ is the number of vertices, 
and $\ell$ is the number of connected components of the digraph,
\blue{cf.~\cite{Feinberg1972}.}

Analogously, we introduce the {\em kinetic-order subspace}
\begin{equation}
\tilde S = \im (\tilde Y I_E) 
\end{equation}
and the {\em kinetic(-order) deficiency} 
\begin{equation}
\tilde \delta
= \dim(\ker \tilde Y \cap \im I_E)
= m - \ell - \dim(\tilde S) .
\end{equation}

A steady state $x \in \R^n_>$ of the ODE system~\eqref{ode2} that fulfills
\begin{equation}
A_k \, x^{\tilde Y} = 0
\end{equation}
is called a positive {\em complex-balanced equilibrium} \blue{(CBE).}

The Jacobian matrix of the right-hand-side of~\eqref{ode2} is given by
\begin{equation} \label{J}
J(x) = Y A_k \diag(x^{\tilde Y}) \, \tilde Y^\trans \diag(x^{-1}) .
\end{equation}
For a generalized mass-action system $(G_k,y,\tilde y)$,
we study whether, for all rate constants $k$, 
every positive CBE $x$ is linearly stable (in its stoichiometric class),
that is, the corresponding Jacobian matrix $J(x)$ is stable on $S$.





\subsection{Previous results on complex-balanced equilibria} \label{sec:previous}

In this work, we consider (sufficient conditions for) the linear stability of CBE.
Previous results concerned existence and uniqueness of CBE
(in every stoichiometric class, for all rate constants)
and necessary conditions for linear stability.

In a first result, we characterized uniqueness of CBE 
via sign vectors of the stoichiometric and kinetic-order subspaces.

\begin{thm}[\cite{MuellerRegensburger2012}, Proposition~3.1] \label{thm:unique}
Let $(G_k,y,\tilde y)$ be a weakly reversible generalized mass-action system.
There exists at most one CBE in every stoichiometric class, for all rate constants, 
if and only if 
$\sign(S) \cap \sign(\tilde S^\perp) \neq \{0\}$.
\end{thm}

For surveys on the uniqueness of equilibria and related injectivity results,
see~\cite{Mueller2016, BanajiPantea2016, FeliuMuellerRegensburger2019}.
In a second result, we characterized existence (for all rate constants, but not in every stoichiometric class).

\begin{thm}[\cite{MuellerRegensburger2014}, Theorem~1] \label{thm:exists}
Let $(G_k,y,\tilde y)$ be a generalized mass-action system.
There exists a CBE, for all rate constants,
if and only if 
$\tilde \delta=0$ and the network is weakly reversible.
\end{thm}

Most importantly, we were able to characterize (combined) existence and uniqueness.

\begin{thm}[\cite{MuellerHofbauerRegensburger2019}, Theorem~45] \label{thm:existsunique}
Let $(G_k,y,\tilde y)$ be a generalized mass-action system.
There exists exactly one CBE in every stoichiometric class, for all rate constants, 
if and only if
$G$ is weakly reversible, ${\delta=\tilde \delta =0}$, $\sign(S) \cap \sign(\tilde S^\perp) = \{0\}$,
and two more conditions on $S$ and $\tilde S$ hold,
cf.~(ii) and (iii) in \cite[Theorem~14]{MuellerHofbauerRegensburger2019}.
\end{thm}

Moreover, we characterized {\em robust} existence and uniqueness via a {\em simple} sign vector condition.

\begin{thm}[\cite{MuellerHofbauerRegensburger2019}, Theorem~46] \label{thm:robust}
Let $(G_k,y,\tilde y)$ be a generalized mass-action system.
There exists exactly one CBE in every stoichiometric class,
for all rate constants and for all small perturbations of the kinetic orders, 
if and only if
$G$ is weakly reversible, $\delta=\tilde \delta =0$, and $\sign(S) \subseteq \sign(\tilde S)^\downarrow$.
\end{thm}

Finally, we obtained first stability results,
for example, on the relation between linear stability and uniqueness
and on the linear stability of \blue{cyclic systems (where the underlying graph is a cycle).}

\begin{thm}[\cite{BorosMuellerRegensburger2020}, Theorem~10] \label{thm:stab_unique}
Let $(G_k,y,\tilde y)$ be a weakly reversible generalized mass-action system. 
If every CBE is linearly stable (in its stoichiometric class), for all rate constants,
then there exists at most one CBE in every stoichiometric class, for all rate constants, 
that is, $\sign(S) \cap \sign(\tilde S^\perp) \neq \{0\}$.
\end{thm}

\begin{thm}[\cite{BorosMuellerRegensburger2020}, Theorem~11] \label{thm:stab_cycle}
Let $(G_k,y,\tilde y)$ be a cyclic generalized mass-action system.
Every CBE is linearly stable (in its stoichiometric class), for all rate constants,
if and only if
$Y \! A_{k=1} \tilde Y^\trans$ is $D$-stable on the stoichiometric subspace.
\end{thm}

\subsubsection*{Stability on a linear subspace}

At this point, we recall standard definitions of linear stability and matrix stability
and their extensions to stability on a linear subspace. 

We say that an equilibrium $x$ is linearly stable in its stoichiometric class $x+S$
if the Jacobian matrix $J(x)$ is stable on $S$. 
Analogously, we say that $x$ is diagonally stable/$D$-stable (in $x+S$) if $J(x)$ is
diagonally stable/$D$-stable on $S$.

A square matrix is \blue{(Hurwitz)} stable if all its eigenvalues have negative real part.
Given a linear subspace $L$,
a square matrix $A$ with $\im A \subseteq L$ is stable on $L$
if all eigenvalues of the linear map $A|_L \colon L \to L$ have negative real part.
(In particular, this implies $\im A = L$.)
A square matrix $A$ is $D$-stable on $L$, if $AD$ is stable on $L$ for all positive diagonal matrices $D\in \pd$.
To define diagonal stability, we require the extension of Lyapunov's Theorem to stability on a linear subspace, 
cf.~\cite[Proposition~5]{BorosMuellerRegensburger2020}:
A square matrix $A$ is stable on $L$ if and only if there exists $P=P^\trans>0$ on $L$
such that $PA+A^\trans P < 0$ on $L$. 
Now, a square matrix $A$ is diagonally stable on $L$ if $P=D\in\pd$,
that is, the positive definite matrix $P$ in Lyapunov's Theorem is actually a positive diagonal matrix $D$.


\subsection{A new decomposition of the graph Laplacian} \label{sec:new}

\blue{
For simplicity, 
let $G_k=(V,E,k)$ be a strongly connected, labeled, simple digraph
and $A_k$ be the corresponding Laplacian matrix.
It is well known that
\begin{equation*}
\ker A_k = \im K_k
\end{equation*}
with a positive vector $K_k \in \R^V_>$ (depending on the rate constants).
The entries of $K_k$ (the {\em tree constants}) can be given explicitly in terms of $k$,
\begin{equation} \label{eq:Kk}
(K_k)_i = \sum_{(V,E') \in T_i} \; \prod_{(j \to j') \in E'} k_{j \to j'} , \quad i \in V ,
\end{equation}
where $T_i$ is the set of directed spanning trees of $G_k$ rooted at vertex $i \in V$ 
(and directed towards the root).
The tree constants $K_k$ correspond to minors of the matrix $-A_k$
which is the content of the matrix-tree theorem (for labeled, directed graphs)~\cite[Theorem~3.6]{Tutte1948}.
For a minimal proof of Eqn.~\eqref{eq:Kk}, see \cite[Lemma~1]{Kandori1993} or \cite[Appendix~A]{Mueller2022}.

In general,}
let $G_k=(V,E,k)$ be a labeled, simple digraph with $\ell$ strongly connected components $G_k^\la=(V^\la,E^\la,k^\la)$, $\la = 1, \ldots, \ell$.
\blue{The corresponding Laplacian matrix $A_k \in \R^{V \times V}$ is block-diagonal with blocks $A_k^\la \in \R^{V^\la \times V^\la}$,
and the vector of tree constants $K_k \in \R^V_>$ has blocks $K_k^\la \in \R^{V^\la}_>$.}

An {\em auxiliary} graph $G_\EE = (V,\EE)$ is a simple directed graph
with $\ell$ components $G_\EE^\la=(V^\la,\EE^\la)$ and $|\EE^\la|=|V^\la|-1$, $\la = 1, \ldots, \ell$.
That is, an auxiliary graph has the same set of vertices as the original graph,
but an arbitrary set of edges \blue{that} form a directed forest.
(In general, an auxiliary graph is not a subgraph of the original graph,
nor are the trees directed towards a root.)
The corresponding incidence matrix is denoted by $I_\EE \in \R^{V \times \EE}$.
\blue{A chain graph is an auxiliary graph with components $G_\EE^\la=(V^\la,\EE^\la)$ of the form $i_1 \to i_2 \to \ldots \to i_{|V^\la|}$.}

Let $A_k \in \R^{V \times V}$ be the Laplacian matrix of the graph $G_k$,
and let $G_\EE$ be an auxiliary graph.
By~\cite[Theorem~4]{Mueller2022},
there exists an invertible, block-diagonal matrix $\Ac \in \R^{\EE \times \EE}$, 
called the {\em core matrix} of the graph Laplacian, such that
\begin{equation} \label{new}
A_k \diag(K_k) = - I_\EE \Ac I_\EE^\trans .
\end{equation}
By~\cite[Theorem~2]{Mueller2022},
if $G_\EE$ is a chain graph,
then $\Ac$ is non-negative with positive diagonal.

A positive CBE $x \in \R^n_>$ is given by $A_k (x^*)^{\tilde Y} = 0$.
Using \eqref{new} and
$\ker(I_\EE) = \ker(\Ac) = \{0\}$, this is equivalent to
\[
I_\EE^\trans \diag(K_k^{-1}) \, x^{\tilde Y} = 0 ,
\]
that is, to the binomial equations
\begin{equation} \label{cbe}
\frac{x^{\tilde y(i')}}{(K_k)_{i'}} - \frac{x^{\tilde y(i)}}{(K_k)_i} = 0 \quad \text{for } (i \to i') \in \EE ,
\end{equation}
cf.~\cite{MuellerRegensburger2014,Mueller2022}.
Given a particular positive CBE $x^*$,
Eqn.~\eqref{cbe} is equivalent to 
\[
\left(\frac{x}{x^*}\right)^{\tilde y(i')} = \left(\frac{x}{x^*}\right)^{\tilde y(i)} \quad \text{for } (i \to i') \in \EE
\]
and further to $(y(i')-y(i))^\trans \ln (x/x^*)=0$ for $(i \to i') \in \EE =0$, that is, to ${(\tilde Y I_\EE)^\trans \ln (x/x^*) = 0}$.
Since $\ker (\tilde Y I_\EE)^\trans = (\im \tilde Y I_\EE)^\perp = (\im \tilde Y I_E)^\perp = \tilde S^\perp$,
the set of all positive CBEs is given by the monomial parametrization $x= x^* \had \e^{\tilde S^\perp}$.

\subsection{Rewriting the Jacobian matrix} \label{sec:rewrite}

Let $x \in \R^n_>$ be a positive CBE.
In the analysis of its stability,
we have to deal with the dependence of the Jacobian matrix $J(x)$ on $x$.
First, we introduce the scaled Jacobian matrix
\begin{equation} \label{Jred}
\JJ(x) = Y A_k \diag(x^{\tilde Y}) \, \tilde Y^\trans ,
\end{equation}
that is, $J(x) = \JJ(x) \diag(x^{-1})$, cf.~Eqn.~\eqref{J},
and the diagonal matrix is taken into account 
by considering the $D$-stability of $\JJ(x)$.
Second, we observe that the entries of $\frac{x^{\tilde Y}}{K_k} \in \R^{V}_>$ are identical on connected components,
cf.~Eqn.~\eqref{cbe}.
Hence, $x^{\tilde Y} = D_V K_k$ with a diagonal matrix $D_V \in \R^{V \times V}$.
In particular, the blocks of $D_V$ (corresponding to the $\ell$ connected components) are positive multiples of the identity matrices $\id^{V_i} \in \R^{V_i \times V_i}$, $i = 1, \ldots, \ell$.
That is,
\[
D_V = 
\begin{pmatrix}
c_1 \, \id^{V_1}  & 0 & \ldots & 0 \\
0 & c_2 \, \id^{V_2} & & \vdots \\
\vdots & & \ddots & 0 \\
0 & \ldots & 0 & c_\ell \, \id^{V_\ell}
\end{pmatrix} 
\]
with $c_i = (c_k(x))_i>0$, $i = 1, \ldots, \ell$, depending on $k$ and $x$.
Hence, $A_k \diag(x^{\tilde Y}) = A_k \diag(K_k) D_V$. 

{\em Decomposition I.}
Given an auxiliary graph $(V,\EE)$
and using the new decomposition of the graph Laplacian, Eqn.~\eqref{new},
we have
\begin{align*}
\JJ(x) &= Y A_k \diag(K_k) D_V \tilde Y^\trans \\
&= - Y I_\EE \Ac I_\EE^\trans D_V \tilde Y^\trans .
\end{align*}
%
Further, since the incidence matrix $I_\EE \in \R^{V \times \EE}$ is block-diagonal,
it holds that
$
I_\EE^\trans D_V = D_\EE I_\EE^\trans
$
with a diagonal matrix $D_\EE \in \R^{\EE \times \EE}$.
In particular, the blocks of $D_\EE$ are positive multiples of the identity matrices $\id^{\EE_i} \in \R^{\EE_i \times \EE_i}$, $i = 1, \ldots, \ell$.
That is,
\[
D_\EE = 
\begin{pmatrix}
c_1 \, \id^{\EE_1}  & 0 & \ldots & 0 \\
0 & c_2 \, \id^{\EE_2} & & \vdots \\
\vdots & & \ddots & 0 \\
0 & \ldots & 0 & c_\ell \, \id^{\EE_\ell}
\end{pmatrix} .
\]
As a result, we obtain the more symmetric form
\begin{equation} \label{Jrewrite}
\JJ(x) = - Y I_\EE \Ac D_{k,\EE}(x) I_\EE^\trans \tilde Y^\trans 
\end{equation}
with the diagonal matrix $D_\EE = D_{k,\EE}(x)$, depending on $k$ and $x$.

{\em Decomposition II.}
Alternatively, we use cycle decomposition of the graph Laplacian.
\blue{In the context of mass-action systems,
this technique has been applied to obtain the foundational result on CBEs~\cite[Theorem~6A]{HornJackson1972},
and it has also been used to study linear stability~\cite{SiegelJohnston2008}.}

Given $G=(V,E)$,
\[
A_k \diag(K_k) = \sum_C \la_{k,C} \, A_C , 
\]
where the sum is over all cycles $C$ in $G$,
$\la_{k,C}>0$ depends on $k$,
and $A_C$ is the Laplacian matrix of $C$ with $k=1 \in \R^E_>$ (all edge labels set to 1),
cf.~\cite[Appendix A, Fact~11]{Mueller2022}.
Hence,
\begin{align*}
\JJ(x) &= Y A_k \diag(K_k) D_V \tilde Y^\trans \\
&= Y \left( \sum_C \la_{k,C} A_C D_V \right) Y^\trans \\
&= Y \left( \sum_C \la_{k,C} \, c_{i(C)} A_C \right) Y^\trans 
\end{align*}
with $c_{i(C)} = (c_k(x))_{i(C)}>0$ corresponding to the connected component~$i$ of the cycle~$C$.
Finally, considering $\la_C = \la_{k,C} \, c_{i(C)}$ as independent ``cycle parameters'',
\begin{equation} \label{eq:cycle}
\JJ = Y \left( \sum_C \la_C A_C \right) \tilde Y^\trans .
\end{equation}
For every CBE, the scaled Jacobian matrix is of this form.
By guaranteeing its $D$-stability 
(for all cycle parameters), 
one shows the linear stability of every CBE 
(for all rate constants).


\subsection{Monomial orders and hyperplane arrangements} \label{sec:orders}


Let $(G_k,y,\tilde y)$ be a generalized mass-action system 
based on a labeled, simple digraph $G_k=(V,E,k)$ and the maps $y,\tilde y$ (the matrices $Y,\tilde Y$).
In particular, let $(G_k,y,\tilde y)$ be weakly reversible (let the components of $G_k$ be strongly connected)
and $x^* \in \R^n_>$ be a positive CBE.

For \blue{fixed} $x \in \R^n_>$,
the values of the monomial \blue{terms} $\frac{x^{\tilde y(i)}}{(K_k)_i} \in \R_>$ with $i \in V$ 
are ordered. 
For simplicity, we first consider a strongly connected graph $G_k=(V,E,k)$ with $m=|V|$.
Obviously, \blue{there is a} {\em total order} 
\[
\frac{x^{\tilde y(i_1)}}{(K_k)_{i_1}} \le \frac{x^{\tilde y(i_2)}}{(K_k)_{i_2}} \le \ldots \le \frac{x^{\tilde y(i_m)}}{(K_k)_{i_m}} .
\]
By Eqn.~\eqref{cbe} 
for the positive CBE $x^* \in \R^n_>$,
the order is equivalent to
\[
\left(\frac{x}{x^*}\right)^{\tilde y(i_1)} \le \left(\frac{x}{x^*}\right)^{\tilde y(i_2)} \le \ldots \le \left(\frac{x}{x^*}\right)^{\tilde y(i_m)} .
\]
The total order can be represented by a chain graph,
\[
i_1 \to i_2 \to \ldots \to i_m.
\]
If the order is non-strict \blue{(if some monomial terms have the same value),}
then the representation by a chain graph (defining a strict order on the vertices) is not unique.
In fact, every total order on the monomial \blue{terms}
can be represented in an {\em equivalent} way by a {\em total preorder} on the vertices.
For example,
if \[\left(\frac{x}{x^*}\right)^{\tilde y(i_1)} < \left(\frac{x}{x^*}\right)^{\tilde y(i_2)} = \left(\frac{x}{x^*}\right)^{\tilde y(i_3)} < \ldots ,\]
then the order can be represented by the graph \[i_1 \to i_2 \rightleftarrows i_3 \to \ldots ,\] defining a preorder.
(Obviously, this is not an auxiliary graph $(V,\EE)$ in the sense of the previous subsection, since $|\EE|>|V|-1$.)

\subsubsection*{Monomial evaluation orders}

For a graph with several components,
we call an order on the entries of $\frac{x^{\tilde Y}}{K_k} \in \R^V_>$ 
and ${(\frac{x}{x^*})^{\tilde Y} \in \R^V_>}$ \blue{that is total within connected components, but does not relate entries in different components,}
a {\em monomial evaluation order}
(since the notion {\em monomial order(-ing)} has a different meaning in algebra).
We represent the order by a chain graph $G_{\EE} = (V,\EE)$
and often just by the set of edges $\EE$.
Explicitly,
$(i \to i') \in \EE$ 
implies $\frac{x^{\tilde y(i)}}{(K_k)_i} \le \frac{x^{\tilde y(i')}}{(K_k)_{i'}}$ and $(\frac{x}{x^*})^{\tilde y(i)} \le (\frac{x}{x^*})^{\tilde y(i')}$
\blue{and that} the vertices $i,i' \in V$ \blue{are} in the same strongly connected component.
If the order is non-strict, then $\EE$ is not unique.

Conversely,
let $G_\EE = (V,\EE)$ be a chain graph.
The subset of $\R^n_>$ with monomial evaluation order represented by $\EE$ is given by
\begin{equation} \label{st}
\begin{aligned}
\ST &= \left\{ x \in \R^n_> \mid \frac{x^{\tilde y(i')}}{(K_k)_{i'}} - \frac{x^{\tilde y(i)}}{(K_k)_i} \ge 0 \text{ for } (i \to i') \in \EE \right\} \\
&= \left\{ x \in \R^n_> \mid I_\EE^T \diag(K_k^{-1}) \, x^{\tilde Y} \ge 0 \right\} .
\end{aligned}
\end{equation}

By Eqn.~\eqref{cbe} 
for the positive CBE $x^* \in \R^n_>$, 
\begin{align*}
\ST 
&= \left\{ x \in \R^n_> \mid \left(\frac{x}{x^*}\right)^{\tilde y(i')} - \left(\frac{x}{x^*}\right)^{\tilde y(i)} \ge 0 \text{ for } (i \to i') \in \EE \right\} \\
&= \left\{ x \in \R^n_> \mid I_\EE^T \left(\frac{x}{x^*}\right)^{\tilde Y} \ge 0 \right\} .
\end{align*}
Further, by the monotonicity of the logarithm,
\begin{align*}
\ST &= \left\{ x \in \R^n_> \mid (\tilde y(i')-\tilde y(i))^T \ln \frac{x}{x^*} \ge 0 \text{ for } (i \to i') \in \EE \right\} \\
&= \left\{ x \in \R^n_> \mid (\tilde Y I_\EE)^T \ln \frac{x}{x^*} \ge 0 \right\} .
\end{align*}

Hence,
\[
x \in \ST
\quad \Leftrightarrow \quad
\ln \frac{x}{x^*} \in \CE 
\]
with the polyhedral cone
\begin{align} \label{ce}
\CE 
&= \left\{ z \in \R^n \mid (\tilde Y I_\EE)^\trans z \ge 0 \right\} ,
\end{align}
which does not depend on $k$.
(Of course, $x^*$ depends on $k$.)
The lineality space of $\CE$ does not even depend on~$\EE$,
\[
\lineal \CE = \ker \, (\tilde Y I_\EE)^\trans = (\im \tilde Y I_\EE)^\perp = (\im \tilde Y I_E)^\perp = \tilde S^\perp .
\]

Given $\EE$, there are two possibilities:
\begin{itemize}
\item
$\CE = \tilde S^\perp$.
Then all defining inequalities of $\CE$ are fulfilled with equality,
and $\ST = x^* \had \e^{\tilde S^\perp}$ equals the set of CBE.
\item
$\CE$ is full-dimensional.
Then, generically, the monomial evaluation order is strict in the interior of $\ST$ (except, of course, if there are identical 
monomials within components) and non-strict on the boundary.
\end{itemize}

A full-dimensional subset $\ST$ is called a {\em stratum},
cf.~\cite{SiegelJohnston2011} in the setting of classical mass-action systems.
This term has also been used for partial orders related to the original graph, rather than to an auxiliary graph,
cf.~\cite{CraciunDickensteinShiuSturmfels2009}.
\blue{Monomial orders are also used to define ``tiers'' in \cite{Anderson2011}.}

As stated above, given $x \in \R^n_>$,
there is a (non-unique) $\EE$ such that $x \in \ST$.
In particular, $\R^n_>$ is a union of strata which intersect only on their boundaries.

\subsubsection*{Central hyperplane arrangements}

Alternatively, total orders on the monomials
can be represented (in an equivalent way) by sign vectors of a linear subspace.

By the monotonicity of the logarithm, the order on the entries of ${(\frac{x}{x^*})^{\tilde Y} \in \R^V_>}$ (within components)
is equivalent to the order on the entries of $\tilde Y^\trans z \in \R^V$ with $z = \ln \frac{x}{x^*}$.
Hence, we consider the set of pairs of vertices (within components),
\begin{equation}
\Omega = \left\{ i \to i' \mid i,i' \in V^\la, \, \la=1,\ldots,\ell \right\} ,
\end{equation}
and the resulting arrangement of central hyperplanes
\[
\tilde h_{i \to i'} = \{ z \in \R^n \mid (\tilde y(i')-\tilde y(i))^T z = 0 \}, \quad (i \to i') \in \Omega .
\]
The central hyperplane arrangement decomposes $\R^n$ into open polyhedral cones called {\em faces};
full dimensional faces are called {\em cells}.
In particular, a cell is the interior of a stratum.

Clearly, $(V,\Omega)$ is the complete graph on $V$ (within components)
with incidence matrix $I_\Omega \in \R^{V \times \Omega}$.
Hence,
\[
\sign \left( (\tilde Y I_\Omega )^\trans z \right) \in \{-,0,+\}^\Omega
\]
determines, for every pair of vertices $(i \to i') \in \Omega$, whether
\[
(\tilde y(i')-\tilde y(i))^T z \lesseqqgtr 0 .
\]
That is, 
$\sign \big( (\tilde Y I_\Omega )^\trans z \big)$ represents the total order on the entries of $\tilde Y^\trans z \in \R^V$
and equivalently on the entries of $(\frac{x}{x^*})^{\tilde Y} \in \R^V_>$ (the monomial evaluation order).
To cover all orders, we introduce the linear subspace 
\begin{equation}
\tilde T = \im (\tilde Y I_\Omega)^\trans \subseteq \R^\Omega .
\end{equation}
Every monomial evaluation order is represented in an {\em equivalent} way by an element of $\sign(\tilde T) \subseteq \{-,0,+\}^\Omega$.
%

In analogy to $\tilde T$, given by the kinetic-order complexes $\tilde Y \in \R^{n \times V}$,
we introduce 
\begin{equation}
T = \im (Y I_\Omega)^\trans \subseteq \R^\Omega , 
\end{equation}
given by the (stoichiometric) complexes $Y \in \R^{n \times V}$.
For an illustration of hyperplane arrangements and the corresponding orders, see Example~\ref{exa:Lotka} below.


\section{Main results} \label{sec:main}

We study the linear stability of CBE.
In the first two subsections,
we use the new decomposition of the graph Laplacian and monomial evaluation orders.
In the third subsection,
we also use cycle decomposition.
As a preparation,
we study the non-existence of other steady states.

\subsection{Non-existence of other steady states} \label{sec:noother}

\blue{
From the notation section,
recall that the total closure of a set of sign vectors $\mathcal{T} \subseteq \{-,0,+\}^n$
is denoted by $\mathcal{T}^\updownarrow$.}

\begin{thm} \label{thm:noother}
Let $(G_k,y,\tilde y)$ be a generalized mass-action system
with a positive CBE (in some stoichiometric class).
If \[ \sign(\tilde T) \subseteq \sign(T)^\updownarrow, \]
then every positive steady state of the dynamical system~\eqref{ode2} is a CBE.
More generally, if \[ \sign \left( (\tilde Y I_\Omega)^\trans (\Sigma(S)) \right) \subseteq \sign(T)^\updownarrow,\]
then every positive steady state in a stoichiometric class with a positive CBE is itself a CBE.
\end{thm}
\begin{proof}
Let $x^* \in \R^n_>$ be a CBE of $\dd{x}{t} = f_k(x) = Y A_k \, x^{\tilde Y}$,
and let $x \in \R^n_>$ {\em not} be a CBE.
We will find $z \in \R^n$ such that $z^\trans f_k(x) \neq 0$, 
and hence $x$ is not a steady state.
In fact,
given an auxiliary graph $(V,\EE)$
and using the new decomposition of the graph Laplacian, 
\[
A_k \diag(K_k) = - I_\EE \Ac I_\EE^\trans ,
\]
cf.~Eqn.~\ref{new},
\begin{align*}
z^\trans f_k(x) &= z^\trans Y A_k \, x^{\tilde Y} \\
&= - z^\trans Y I_\EE \Ac I_\EE^\trans \diag(K_k^{-1}) \, x^{\tilde Y} \\
&= - a^\trans \Ac \,a' ,
\end{align*}
and
\begin{align*}
a &= (Y I_\EE)^\trans z , \\
a' &= I_\EE^\trans \diag(K_k^{-1}) \, x^{\tilde Y} 
\end{align*}
fulfill
\[
a,a' \ge 0 \quad \text{and} \quad a \cdot a' > 0 .
\]
Since the core Laplacian matrix $\Ac$ is nonnegative with positive diagonal,
it follows that $z^\trans f_k(x)<0$, and $x$ is not an equilibrium.
It remains to identify a vector $z$ and a chain graph $(V,\EE)$,
representing a total order on the vertices.

By Eqn.~\eqref{cbe} 
for the positive CBE $x^* \in \R^n_>$,
the order on the entries of $\frac{x^{\tilde Y}}{K_k} \in \R^V_>$
is equivalent to the order on the entries of $(\frac{x}{x^*})^{\tilde Y} \in \R^V_>$.
By the monotonicity of the logarithm,
it is further equivalent to the order on the entries of $\tilde Y^\trans \tilde z \in \R^V$ with $\tilde z = \ln \frac{x}{x^*} \in \R^n$.
Hence,
\[
\tilde \tau = \sign \left( I_\Omega^\trans \diag (K_k^{-1}) \, x^{\tilde Y} \right) 
= \sign \left( I_\Omega^\trans \, \tilde Y^\trans \tilde z \right) = \sign \left( ( \tilde Y I_\Omega)^\trans \tilde z \right) . 
\]
Since $x$ is not a CBE, $\tilde \tau \in \sign(\tilde T)$ is nonzero.
If $\sign(\tilde T) \subseteq \sign(T)^\updownarrow$,
then there exists a nonzero $\tau \in \sign(T)$ 
and a corresponding $z \in \R^n$ with
\[
\tau = \sign \left( ( Y I_\Omega)^\trans z \right) 
\]
such that $\tilde \tau \le \tau$ or $\tilde \tau \ge \tau$ (and hence $\tau \hads \tilde \tau > 0$).

By Lemma~\ref{lem:exist} below (for $z,\tilde z$ and nonzero $\tau,\tilde \tau$),
there exists a chain graph $G=(V,\EE)$ with incidence matrix $I_\EE \in \R^{V \times \EE}$ such that
\begin{align*}
a &= (Y I_\EE)^\trans z , \quad \text{(as above)} \\
\tilde a &= (\tilde Y I_\EE)^\trans \tilde z .
\end{align*}
fulfill
\[
a, \tilde a \ge 0 \quad \text{and} \quad a \cdot \tilde a > 0 .
\]
As shown above,
$
\sign \big( I_\Omega^\trans \diag (K_k^{-1}) \, x^{\tilde Y} \big) 
= \sign \big( ( \tilde Y I_\Omega)^\trans \tilde z \big) ,
$
and hence also
\[
\sign(a') 
= \sign \left( I_\EE^\trans \diag (K_k^{-1}) \, x^{\tilde Y} \right) 
= \sign \left( ( \tilde Y I_\EE)^\trans \tilde z \right) 
= \sign(\tilde a) 
\]
such that $a, a' \ge 0$ and $a \cdot a' > 0$, as claimed.

If $x$ is in the stoichiometric class of $x^*$, then $x=x^*+v$ with $v \in S$,
and $\sign(\tilde z) = \sign (\ln \frac{x^*+v}{x^*}) = \sign(v)$, by the monotonicity of the logarithm.
Hence, $\tilde z \in \Sigma(S)$ and $\tilde \tau \in \sign ( (\tilde Y I_\Omega)^\trans (\Sigma(S)) ) =: \mathcal{\tilde T}$,
and it is sufficient to require $\mathcal{\tilde T} \subseteq \sign(T)^\updownarrow$.
\end{proof}

\begin{lem} \label{lem:exist}
Let $(G_k,y,\tilde y)$ be a generalized mass-action system 
based on a labeled, simple digraph $G_k=(V,E,k)$ and the maps $y,\tilde y$ (the matrices $Y,\tilde Y$).
Further, let $I_\Omega \in \R^{V \times \Omega}$ be the incidence matrix of $(V,\Omega)$,
the complete graph (within components).

Now, let $z, \tilde z \in \R^n$ and $\tau = \sign \big( (Y I_\Omega)^\trans z \big)$,
$\tilde \tau = \sign \big( (\tilde Y I_\Omega)^\trans \tilde z \big)$ be nonzero.
If $\tau \hads \tilde \tau \ge 0$,
then there exists a chain graph $G=(V,\EE)$ with incidence matrix $I_\EE \in \R^{V \times \EE}$ such that
\[
a = (Y I_\EE)^\trans z \ge 0, \quad
\tilde a = (\tilde Y I_\EE)^\trans \tilde z \ge 0,
\quad \text{and} \quad
a,\tilde a \neq 0 .
\]
If further $\tau \le \tilde \tau$ or $\tau \ge \tilde \tau$,
then
\[
a \cdot \tilde a > 0 .
\]
\end{lem}
\begin{proof}
The sign vectors $\tau, \tilde \tau \in \{-,0,+\}^\Omega$ represent total orders on the entries of $Y^\trans z$, $\tilde Y^\trans \tilde z \in \R^V$
and corresponding total preorders $\le_1, \le_2$ on $V$,
cf.~Appendix~\ref{app:order}.

If $\tau \hads \tilde \tau \ge 0$,
then $\le_1, \le_2$ are harmonious
in the sense that $i <_1 j$ implies $i \le_2 j$ (and vice versa) or, 
explicitly, $( y(j)-y(i) )^\trans z > 0$ implies $( \tilde y(j)- \tilde y(i) )^\trans \tilde z \ge 0$, for all $i,j \in V$.
By Lemma~\ref{lem:order} in Appendix~\ref{app:order}, 
there exists a total strict order $<$ on $V$ such that $i < j$ implies $i \le_1 j$ and $i \le_2 j$.
Clearly, $<$ can be represented by a chain graph $(V,\EE)$.
For $a=(Y I_\EE)^\trans z$, $\tilde a =(\tilde Y I_\EE)^\trans \tilde z \in \R^\EE$,
this implies $a,\tilde a \ge 0$.
Since $\tau, \tilde \tau \neq 0$, not all entries of $Y^\trans z$, $\tilde Y^\trans \tilde z$ are equal,
and $a, \tilde a \neq 0$.

Finally, if $\tau \le \tilde \tau$, consider $(i \to j) \in \EE$
such that $\tau_{i \to j} = +$, that is, $a_{i \to j} = ( y(j)-y(i) )^\trans z) > 0$.
Then, also $\tilde \tau_{i \to j} = +$, that is, $\tilde a_{i \to j} = ( \tilde y(j)-\tilde y(i) )^\trans z) > 0$, 
and $a \cdot \tilde a > 0$.
If $\tau \ge \tilde \tau$, consider $(i \to j) \in \EE$
such that $\tilde \tau_{i \to j} = +$,
and use the same argument.
\end{proof}


\subsection{Linear stability of complex-balanced equilibria} \label{sec:stab}

%

Theorem~\ref{thm:stab_unique} provides necessary conditions
for the linear stability of complex balanced-equilibria (for all rate constants),
namely, their uniqueness (in every stoichiometric class, for all rate constants);
technically, $\sign(S) \cap \sign(\tilde S^\perp) = \{0\}$.
%

Using the new decomposition of the graph Laplacian,
we first provide a new proof of Theorem~\ref{thm:stab_unique}. 

\begin{proof}[Proof of Theorem~\ref{thm:stab_unique}]
Assume $\sign(S) \cap \sign(\tilde S^\perp) \neq \{0\}$,
that is, there exists a nonzero $\tau \in \sign(S) \cap \sign(\tilde S^\perp)$ 
and hence nonzero $v \in S$, $\tilde v \in \tilde S^\perp$, and $x \in \R^n_>$ such that $v = \diag(x) \, \tilde v$.

Since the underlying network is weakly reversible,
$x \in \R^n_>$ is a CBE for some rate constants $k$, cf.~\cite[Lemma 1, Proof]{MuellerRegensburger2014}.
For an auxiliary graph $(V,\EE)$,
\begin{align*}
J(x) &= Y A_k \diag(x^{\tilde Y}) \, \tilde Y^\trans \diag(x^{-1}) \\
&= - Y I_\EE \Ac I_\EE^\trans \diag(K_k^{-1}) \diag(x^{\tilde Y}) \, \tilde Y^\trans \diag(x^{-1}) \\
&= - Y I_\EE \Ac I_\EE^\trans D_V \tilde Y^\trans \diag(x^{-1}) \\
&= - Y I_\EE \Ac D_{k,\EE}(x) I_\EE^\trans \tilde Y^\trans \diag(x^{-1})
\end{align*}
with $D_{k,\EE}(x) \in \pd$,
see the discussion before Eqn.~\eqref{Jrewrite}.
Now, 
\begin{align*}
J(x) \, v &= \ldots \, (\tilde Y I_\EE)^\trans \diag(x^{-1}) \, v \\
&= 0 ,
\end{align*}
since
$\ker (\tilde Y I_\EE)^\trans = \im \tilde Y I_\EE = \im \tilde Y I_E = \tilde S$
and
${\diag(x^{-1}) \, v = \tilde v \in \tilde S}$.

To summarize, there are rate constants $k$ and a CBE $x$ such that $J(x)$ has a zero eigenvalue on $S$,
that is, $x$ is not linearly stable (in its stoichiometric class).
\end{proof}


A main result of this work, Proposition~\ref{pro:Pmatrix} below,
provides a sufficient condition for the linear stability of CBE (for all rate constants).
It guarantees that the negative of the Jacobian matrix is a $P$-matrix.
For planar systems, this is equivalent to diagonal stability.
For general systems, the $P$-matrix property and sign-symmetry together guarantee stability.

\begin{dfn}
A matrix $A \in \R^{n \times n}$ is a {\em $P$-matrix} if all its principal minors are positive.
\end{dfn}

\begin{fac}[\cite{FiedlerPtak1962}, (3,3) Theorem] \label{fac:fiedler}
A matrix $A \in \R^{n \times n}$ is a $P$-matrix
if and only if, for all nonzero $x \in \R^n$, there is $D \in \pd$ such that $Dx \cdot Ax = x^\trans \! DA x> 0$.
\end{fac}

\begin{dfn}
A matrix $A \in \R^{n \times n}$ is {\em sign symmetric} if its minors satisfy
$A[\al|\be] \, A[\be|\al] \ge 0$ for all $\al, \be \subset \{1,\ldots,n\}$ with $|\al|=|\be|$,
that is, if pairs of symmetrically placed minors do not have opposite signs.
\end{dfn}

\begin{thm}[\cite{Carlson1974}, Theorem] \label{thm:old:n}
If $A \in \R^{n \times n}$ is sign-symmetric
and $-A$ is a $P$-matrix, then $A$ is stable.
\end{thm}

\begin{thm}[\cite{Cross1978}, Fact 2.8.1] \label{thm:old:2}
A matrix $A \in \R^{2 \times 2}$ is diagonally stable 
if and only if $-A$ is a $P$-matrix.
\end{thm}

Clearly, these results are non-trivial only for matrices with full rank.
They motivate Proposition~\ref{pro:Pmatrix}.
Its proof uses the new decomposition of the graph Laplacian,
monomial evaluation orders,
and the characterization of a $P$-matrix via positive diagonal matrices.

\begin{pro} \label{pro:Pmatrix}
Let $(G_k,y,\tilde y)$ be a generalized mass-action system
with $S = \tilde S = \R^n$
and 
\[
\sign \left( (\tilde Y I_\Omega)^\trans (O) \right) 
\subseteq 
\sign \left( (Y I_\Omega)^\trans (\overline O) \right)^\updownarrow, 
\quad \text{for all open orthants } 
O \subset \R^n .
\]
Then, for every positive CBE,
the negative of the Jacobian matrix is a $P$-matrix.
\end{pro}
\begin{proof}
Let $x \in \R^n_>$ be a CBE
and $J(x) \in \R^{n \times n}$ be the corresponding Jacobian matrix,
\[
J(x) = Y A_k \diag(x^{\tilde Y}) \, \tilde Y^\trans \diag(x^{-1}) ,
\]
cf.~Eqn.~\eqref{J}.
We will show that the negative of the scaled Jacobian matrix
\[
\JJ(x) = Y A_k \diag(x^{\tilde Y}) \, \tilde Y^\trans ,
\] 
cf.~Eqn.~\eqref{Jred}, is a $P$-matrix.
(Then also $-J(x) = -\JJ(x) \diag(x^{-1})$ is a $P$-matrix,
since multiplication with a positive diagonal matrix does not affect the signs of principal minors.)

Using Fact~\ref{fac:fiedler}, for every nonzero $\tilde z \in \R^n$,
we will find a diagonal matrix $D' \in \pd$
such that 
\[
- \tilde z^\trans D' \JJ(x) \, \tilde z > 0 .
\]
Given an auxiliary graph $(V,\EE)$
and using the new decomposition of the graph Laplacian, 
\[
A_k \diag(K_k) = - I_\EE \Ac I_\EE^\trans ,
\]
cf.~Eqn.~\ref{new}, 
we can rewrite the scaled Jacobian matrix in the more symmetric form
\begin{align*}
\JJ(x) &= - Y I_\EE \Ac I_\EE^\trans \diag(K_k^{-1}) \diag(x^{\tilde Y}) \, \tilde Y^\trans \\
&= - Y I_\EE \Ac I_\EE^\trans D_V \tilde Y^\trans \\
&= - Y I_\EE \Ac D_{k,\EE}(x) I_\EE^\trans \tilde Y^\trans 
\end{align*}
with $D_{k,\EE}(x) \in \pd$,
see the discussion before Eqn.~\eqref{Jrewrite}.
Hence, we have
\begin{align*}
- \tilde z^\trans D' \JJ(x) \, \tilde z &= \tilde z^\trans D' \,Y I_\EE \Ac D_{k,\EE}(x) I_\EE^\trans \tilde Y^\trans \tilde z \\
&= z^\trans Y I_\EE \Acc (\tilde Y I_\EE)^\trans \tilde z \\
&= a^\trans \Acc \, \tilde a ,
\end{align*}
\blue{where} we introduced $z = D' \tilde z$ (in the same open orthant as $\tilde z$) and $\Acc = \Ac D_{k,\EE}(x)$.
Further, we introduced
\begin{align*}
a &= (Y I_\EE)^\trans z , \\
\tilde a &= (\tilde Y I_\EE)^\trans \tilde z .
\end{align*}
As stated above, for every nonzero $\tilde z \in \R^n$,
we will find $D' \in \pd$
such that $-\tilde z^\trans D' \JJ(x) \, \tilde z > 0$.
Ultimately,
we will find $z=D'\tilde z$ and a chain graph $(V,\EE)$
such that $a^\trans \Acc \, \tilde a > 0$.

Let $O \subset \R^n$ be an open orthant and $\tilde z \in O$ be nonzero.
By assumption, $\tilde S=\R^n$,
and we have $\ker (\tilde Y I_\Omega)^\trans = (\im (\tilde Y I_\Omega) )^\perp = \tilde S^\perp = \{0\}$.
Hence, $(\tilde Y I_\Omega)^\trans \tilde z \neq 0$
and further 
$
\tilde \tau = \sign \big( ( \tilde Y I_\Omega)^\trans \tilde z \big) \neq 0
$.
Obviously, $\tilde \tau \in \sign \big( (\tilde Y I_\Omega)^\trans (O) \big) =: \mathcal{\tilde T}$.

Analogously,
let $\mathcal{T} := \sign \big( (Y I_\Omega)^\trans (O) \big)$
and $\mathcal{T'} := \sign \big( (Y I_\Omega)^\trans (\overline{O}) \big)$,
where $\overline{O}$ is the corresponding closed orthant.
Clearly, $\mathcal{T} \subseteq \mathcal{T'}$.

By assumption, $ \mathcal{\tilde T} \subseteq \mathcal{T'}^\updownarrow$.
It remains to {\bf distinguish two cases}:
either (i) there exists a nonzero $\tau \in \mathcal{T}$
such that $\tilde \tau \le \tau$ or $\tilde \tau \ge \tau$, 
or (ii) there is no such $\tau \in \mathcal{T}$,
but there exists a nonzero $\tau' \in \mathcal{T'}$ such that $\tilde \tau \le \tau'$ or $\tilde \tau \ge \tau'$.

(i) For $\tau \in \mathcal{T}$,
there exists a corresponding $z \in O$ 
with
$
\tau = \sign \left( ( Y I_\Omega)^\trans z \right)
$.
Clearly, $z = D' \, \tilde z$ for some $D' \in \pd$.

By Lemma~\ref{lem:exist} (for $z,\tilde z$ and nonzero $\tau,\tilde \tau$),
there exists a chain graph $G=(V,\EE)$ with incidence matrix $I_\EE \in \R^{V \times \EE}$ such that
\[
a, \tilde a \ge 0 \quad \text{and} \quad a \cdot \tilde a > 0 
\]
and hence
\[
a^\trans \Acc \, \tilde a > 0 ,
\]
since the core Laplacian matrix $\Ac$ of a chain graph (and hence the matrix $\Acc$) is nonnegative with positive diagonal.

(ii) For $\tau' \in \mathcal{T'}$,
there exists a corresponding $z' \in \bd O$, on the boundary of the orthant,
with $\tau' = \sign \left( ( Y I_\Omega)^\trans z' \right)$.
Clearly, there is no $D' \in \pd$ such that $z' = D' \, \tilde z$.
Still, there exist sequences $D'_n \in \pd$ and $z_n = D'_n \, \tilde z \to z'$ with $\sign \left( ( Y I_\Omega)^\trans z_n \right) = \tau \in \mathcal{T}$.
Thereby, $\tau > \tau'$,
since a perturbation of $z' \in \bd O$ to $z_n \in O$ leads to $\tau'_{i \to i'}=0$ and $(\tau_n)_{i \to i'} \neq 0$ for some $(i \to i') \in \Omega$.
This also implies $\tilde \tau > \tau'$,
since $\tilde \tau \le \tau'$ yields $\tilde \tau < \tau$, that is, case~(i).

By Lemma~\ref{lem:exist} (for $z',\tilde z$ and nonzero $\tau',\tilde \tau$),
there exists a chain graph $G=(V,\EE)$ with incidence matrix $I_\EE \in \R^{V \times \EE}$ such that
\[
a' = (Y I_\EE)^\trans z' \ge 0, \quad
\tilde a = (\tilde Y I_\EE)^\trans \tilde z \ge 0, \quad
a',\tilde a \neq 0 ,
\quad \text{and} \quad
a' \cdot \tilde a > 0 ,
\]
but
\[
a_n = (Y I_\EE)^\trans z_n \not\ge 0 .
\]
Still, $a_n \to a'$ and hence $(a_n)_{i \to i'} \to 0$ if $(a_n)_{i \to i'} < 0$ for $(i \to i') \in \EE$.
Further, there is $(i \to i') \in \EE$ with $a'_{i \to i'} , \tilde a_{i \to i'} > 0$.
Hence, also $(a_n)_{i \to i'} > 0$ and $(a_n)_{i \to i'} \to a'_{i \to i'} > 0$
such that
\[
a_n^\trans \, \Acc \, \tilde a > 0 ,
\]
for $n$ large enough, which completes the argument.
\end{proof}

Theorems~\ref{thm:old:n} and~\ref{thm:old:2} together with Proposition~\ref{pro:Pmatrix} imply our main stability results (for full-dimensional stoichiometric subspace).

\begin{thm} \label{thm:n}
Let $(G_k,y,\tilde y)$ be a generalized mass-action system with $S=\tilde S=\R^n$.
If the sign condition in Proposition~\ref{pro:Pmatrix} holds
and the Jacobian matrix is sign-symmetric, 
then, for all rate constants, every positive CBE is linearly stable.
\end{thm}

\begin{thm} \label{thm:2}
Let $(G_k,y,\tilde y)$ be a planar generalized mass-action system with $S=\tilde S=\R^2$.
If the sign condition in Proposition~\ref{pro:Pmatrix} holds
then, for all rate constants,
every positive CBE is diagonally stable.
\end{thm}

\begin{rem}
The conditions of Proposition~\ref{pro:Pmatrix} (and hence of Theorems~\ref{thm:n} and \ref{thm:2})
involve $S$ and $\tilde S$, $Y$ and $\tilde Y$, as well as $I_\Omega$,
the incidence matrix of $(V,\Omega)$,
the graph that is complete within the components of $(V,E)$.
This graph depends only on the (partition of the vertex set $V$ into) components of $(V,E)$,
but not on the edge set $E$. 

Hence, the conclusions hold for all generalized mass-action systems
(with the given stoichiometric and kinetic-order complexes)
that are based on a graph $(V,E')$
with the same (partition of the vertex set $V$ into) components,
but arbitrary edge set $E'$.
\end{rem}


Finally, we illustrate the proof of Proposition~\ref{pro:Pmatrix} in an example.

\begin{exa} \label{exa:Lotka}
We study the generalized Lotka reactions~\cite{lotka:1910},
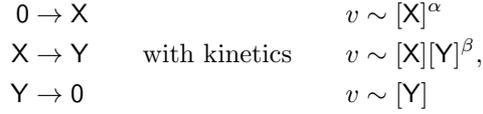
\begin{alignat*}{5}
\ce{0} &\to \ce{X} && v \sim [\ce{X}]^\al \\[-.4ex]
\ce{X} &\to \ce{Y} & \qquad \text{with kinetics} \qquad & v \sim [\ce{X}] [\ce{Y}]^\be , \\ 
\ce{Y} &\to \ce{0} && v \sim [\ce{Y}]
\end{alignat*}
for $\al,\be\ge0$.
In the classical Lotka-Volterra system~\cite{lotka:1920:a,lotka:1920:b}, $\al=\be=1$.

The dynamical system can be specified as a network with MAK:
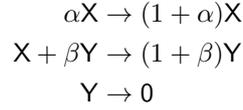
\begin{align*}
\al \ce{X} &\to (1+\al) \ce{X} \\
\ce{X} + \be \ce{Y} &\to (1+\be) \ce{Y} \\ 
\ce{Y} &\to \ce{0} 
\end{align*}
For $\al,\be>0$, the network has $6$ complexes,
is not weakly reversible, 
and has $\delta=6-3-2=1$.

\newcommand{\sst}[1]{{\scriptstyle #1}}

However, the dynamical system can also be specified as a network with GMAK:
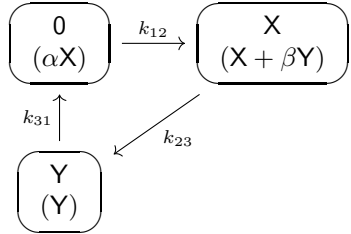
\[
\begin{tikzcd}[ampersand replacement=\&]
\mbox{\ovalbox{$\begin{array}{c} \ce{0} \\ (\al\ce{X}) \end{array}$}} \arrow[r,"k_{12}"] \& \mbox{\ovalbox{$\begin{array}{c} \ce{X} \\ (\ce{X}+\be\ce{Y}) \end{array}$}} \arrow[ld,xshift=+0ex,"k_{23}"] \\
\mbox{\ovalbox{$\begin{array}{c} \ce{Y} \\ (\ce{Y}) \end{array}$}} \arrow[u,"k_{31}"]
\end{tikzcd}
\]
The network has $3$ complexes ($V=\{1,2,3\}$), is weakly reversible, and has $\delta = 3-1-2=0$, $\tilde \delta = 3-1-2=0$ (generically).
In fact,
\[
Y = \bordermatrix{ & \sst{1} & \sst{2} & \sst{3} \cr
\sst{\ce{X}} & 0 & 1 & 0 \cr 
\sst{\ce{Y}} & 0 & 0 & 1
} ,
\quad
\tilde Y = \begin{pmatrix} \al & 1 & 0 \\ 0 & \be & 1 \end{pmatrix} ,
\]
and hence 
\begin{gather*}
S = \im \begin{pmatrix} 1 & 0 \\ 0 & 1 \end{pmatrix} ,
\quad
\tilde S = \im \begin{pmatrix} 1-\al & -\al \\ \be & 1 \end{pmatrix} ,
\end{gather*}
$\dim S=2$, $\dim \tilde S = 2$ (if $1-\al+\al\be\neq0$).
Further,
\[
I_\Omega = \bordermatrix{ & \sst{12} & \sst{13} & \sst{23} \cr 
\sst{1} & -1 & -1 & 0 \cr 
\sst{2} & 1 & 0 & -1 \cr 
\sst{3} & 0 & 1 & 1 } ,
\]
thereby considering only ordered pairs of vertices,
and hence
\begin{gather*}
Y I_\Omega = \begin{pmatrix} \red{1} & \trueblue{0} & \green{-1} \\ \red{0} & \trueblue{1} & \green{1} \end{pmatrix} ,
\quad
\tilde Y I_\Omega = \begin{pmatrix} \red{1-\al} & \trueblue{-\al} & \green{-1} \\ \red{\be} & \trueblue{1} & \green{1-\be} \end{pmatrix} .
\end{gather*}
We provide a sketch of the hyperplane arrangements (in $\R^2$) given by $Y I_\Omega$ and $\tilde Y I_\Omega$ (for $0<\al,\be<1$).

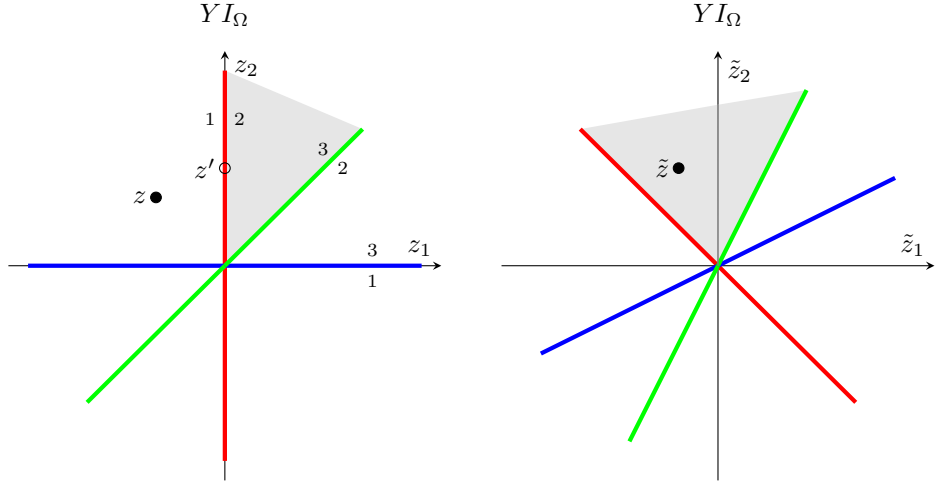
\begin{tikzpicture}
\begin{axis}[width=0.6\textwidth,height=0.6\textwidth,
    axis lines=middle,xtick=\empty,ytick=\empty,
    xmin=-1.1,xmax=1.1,ymin=-1.1,ymax=1.1,samples=10,
    xlabel={$z_1$},ylabel={$z_2$},title={$Y I_\Omega$}]
    \addplot[name path = A, red, ultra thick, domain=0:1](0,x);
    \addplot[red, ultra thick, domain=-1:0](0,x);
    \addplot[blue, ultra thick, domain=-1:1]{0};
    \addplot[name path = C, green, ultra thick, domain=0:.7]{x};
    \addplot[green, ultra thick, domain=-.7:0]{x};
    \addplot [gray!20, opacity=0.5] fill between [of = A and C, soft clip={}];
    \addplot[mark=*] coordinates {(-.35,.35)};
    \node [left] at (-.35,.35) {$z$};
    \addplot[mark=o] coordinates {(0,.5)};
    \node [left] at (0,.5) {$z'$};
    \node [left] at (0,.75) {$\scriptstyle 1$};\node [right] at (0,.75) {$\scriptstyle 2$};
    \node [above] at (.75,0) {$\scriptstyle 3$};\node [below] at (.75,0) {$\scriptstyle 1$};
    \node [left] at (.575,.6) {$\scriptstyle 3$};\node [right] at (.525,.5) {$\scriptstyle 2$};
\end{axis}
\end{tikzpicture}
\qquad
\begin{tikzpicture}
\begin{axis}[width=0.6\textwidth,height=0.6\textwidth,
    axis lines=middle,xtick=\empty,ytick=\empty,
    xmin=-1.1,xmax=1.1,ymin=-1.1,ymax=1.1,samples=10,
    xlabel={$\tilde z_1$},ylabel={$\tilde z_2$},title={$\tilde Y I_\Omega$}]
    \addplot[name path = A, red, ultra thick, domain=-.7:0]{-x};
    \addplot[red, ultra thick, domain=0:.7]{-x};
    \addplot[blue, ultra thick, domain=-.9:.9]{x/2};
    \addplot[name path = C, green, ultra thick, domain=0:.9](x/2,x);
    \addplot[green, ultra thick, domain=-.9:0](x/2,x);
    \addplot[name path = HL, gray, domain=0:1](-x/1000,.825*x);
    \addplot[name path = HR, gray, domain=0:1](x/1000,.825*x);
    \addplot[gray!20, opacity=0.5] fill between [of = A and HL, soft clip={}];
    \addplot[gray!20, opacity=0.5] fill between [of = C and HR, soft clip={}];
    \addplot[mark=*] coordinates {(-.2,.5)};
    \node [left] at (-.2,.5) {$\tilde z$};
\end{axis}
\end{tikzpicture}

We first consider the (stoichiometric) complexes $Y \in \R^{2 \times V}$.
For $z \in \R^2$, the components of $Y^\trans z \in \R^V$ (that is, the projections of the complexes $y(i)$ for $i \in V$ on $z$) are totally ordered.
The order can be represented (in an equivalent way) by a preorder on $V$ or, alternatively,
by the sign vector $\tau = \sign((Y I_\Omega)^\trans z) \in \{-,0,+\}^\Omega$, that is, by a face of the hyperplane arrangement.
A full dimensional face is called a cell.

For example, the red line/hyperplane is given by its normal vector $y(2)-y(1) = {1 \choose 0} - {0 \choose 0} = {\red{1} \choose \red{0}}$.
To its right, $y(2)^\trans z > y(1)^\trans z$, as indicated by the corresponding vertices 1 and 2.
The red half-line from $0$ to ${0\choose\infty}$ corresponds to $y(1)^\trans z = y(2)^\trans z > y(3)^\trans z$,
represented by the preorder $1 \rightleftarrows 2 \to 3$ or, alternatively, by the sign vector $\tau = (0,+,+)^\trans$.
The gray cell corresponds to $1 \to 2 \to 3$ or, alternatively, to $\tau = (+,+,+)^\trans$.

Analogously, the kinetic-order complexes $\tilde Y \in \R^{2 \times V}$ determine the second hyperplane arrangement.

{\em (Diagonal) stability of CBE.}
Generically, $S = \tilde S = \R^2$, 
and it remains to check
\[
\sign \left( (\tilde Y I_\Omega)^\trans (O) \right) 
\subseteq 
\sign \left( (Y I_\Omega)^\trans (\overline O) \right)^\updownarrow ,
\quad \text{ for all open orthants } O \subset \R^2 . 
\]
We first consider the open upper left orthant $O$,
in particular, we consider $\tilde z \in O$ with $\tilde \tau = \sign((\tilde Y I_\Omega)^\trans \tilde z) = (+,+,+)^\trans$ in the gray cell (of the hyperplane arrangement $\tilde Y I_\Omega$).
However, for all $z \in O$ in the given orthant, $\tau = \sign((Y I_\Omega)^\trans z) = (-,+,+)^\trans$
such that neither $\tau \le \tilde \tau$ nor $\tau \ge \tilde \tau$.
Still, there is 
$z' \in \bd O$ on the boundary of the orthant
with $\tau' = \sign((Y I_\Omega)^\trans z') = (0,+,+)^\trans$ and $\tau' \le \tilde \tau$.
A similar argument applies for all orthants (of $\R^2$) and all faces (of the hyperplane arrangement $\tilde Y I_\Omega$), 
and hence the (unique) positive CBE is diagonally stable,
for all rate constants.

Theorem~\ref{thm:stab_cycle} for cyclic networks
or the explicit analysis in \cite{BorosHofbauerMueller2017} guarantee linear stability even for $0 \le \al,\be \le 1$ except $(\al,\be)=(1,1)$,
the classical Lotka-Volterra system, and $(\al,\be)=(1,0)$. 
Our sufficient conditions differ slightly from the previous equivalent conditions, for two reasons:
First, they guarantee diagonal (not just linear) stability.
Second, and most importantly,
they guarantee stability for all networks with given complexes $Y$ and $\tilde Y$,
that is, for all graphs with three vertices (in one component), 
for example, a reversible chain, a partially reversible cycle (in one or the other direction), or the complete graph.

{\em No other steady states.}
Obviously, $\ker Y I_\Omega = \ker \tilde Y I_\Omega = \im (1,-1,1)^\trans$.
Hence $\im (Y I_\Omega)^\trans = \im (\tilde Y I_\Omega)^\trans$
and further
$\sign(\im (Y I_\Omega)^\trans)=\sign(\im (\tilde Y I_\Omega)^\trans)$,
that is,
the two hyperplane arrangements are isomorphic.
By Theorem~\ref{thm:noother},
every positive steady state is a CBE.
(Of course, the result already follows from $\delta=0$.)
\phantom{} \hfill $\blacksquare$
\end{exa}


\subsection{Stability on the stoichiometric subspace} \label{sec:linear}

In case $S \neq \R^n$ (that is, $1 \le \dim(S)<n$), 
one has to consider the linearization $\dd{v}{t} = J(x) \, v$ on $S$ (that is, for $v \in S$).
However, standard stability results are formulated for matrices with full rank.
We outline two alternative strategies.

\subsubsection*{Strategy I}
One chooses an orthonormal basis for $S$, that is, 
$B \in \R^{n \times s}$ with $s=\dim S$, $S=\im B$, and $B^\trans B = \id$. 
Then, every $v \in S$ can be written as $v = B \, z$ with $z \in \R^s$,
and one considers 
\[
\dd{z}{t} = B^\trans \! J(x) B \, z
\] 
in $\R^s$. 
\blue{The ``reduced Jacobian matrix'' $B^\trans \! J(x) B \, z$ has been used in many studies of injectivity/multistationarity;
for references, see e.g.~\cite[Section~2.2]{BanajiPantea2016}.}

\blue{Now,}
we can \blue{extend} Proposition~\ref{pro:Pmatrix} to $S \neq \R^n$.
Technically, we guarantee that the negative of $B^\trans \! \JJ(x) D B$
is a $P$-matrix, for all $D \in \pd$.
\blue{We note that the \blue{extended} results hold for any orthonormal basis matrix $B$,
however, the crucial sign condition depends on $B$.}
%
%
%

\begin{pro} \label{pro:S}
Let $(G_k,y,\tilde y)$ be a generalized mass-action system with $s = \dim S$,
\[
\sign(S) \cap \sign(\tilde S^\perp) = \{0\} ,
\quad \text{and}, \quad \text{for all open orthants } O \subset \R^s ,
\]
\[
\sign \left( (\tilde Y I_\Omega)^\trans (\Sigma( B(O) )) \right) 
\subseteq 
\sign \left( (Y I_\Omega)^\trans (B(\overline O)) \right)^\updownarrow , 
\]
where $S =\im B$ with an orthonormal basis $B \in \R^{n \times s}$.

Then, for every positive CBE $x$,
the negative of $B^\trans \! J(x) B$ is a $P$-matrix.
\end{pro}

\begin{proof}
We proceed as in the proof of Proposition~\ref{pro:Pmatrix}
and point out the required modifications.

Let $x \in \R^n_>$ be a CBE
and $J(x) \in \R^{n \times n}$ be the corresponding Jacobian matrix.
To show that the negative of $B^\trans \! J(x) B$ is a $P$-matrix,
we recall $J(x) = \JJ(x) \diag(x^{-1})$
and show that the negative of
$
B^\trans \! \JJ(x) D B
$
is a $P$-matrix for all $D \in \pd$.
Indeed,
for every $\tilde z \in \R^s$,
we find a positive diagonal matrix $D' \in \pd$
such that, 
\[
- \tilde z^\trans D' B^\trans \! \JJ(x) D B \, \tilde z > 0 .
\]
Given an auxiliary graph $(V,\EE)$
and using the new decomposition of the graph Laplacian, 
\[
A_k \diag(K_k) = - I_\EE \Ac I_\EE^\trans ,
\]
we again rewrite the scaled Jacobian matrix in the more symmetric form
\begin{align*}
\JJ(x) &= - Y I_\EE \Ac I_\EE^\trans \diag(K_k^{-1}) \diag(x^{\tilde Y}) \, \tilde Y^\trans \\
&= - Y I_\EE \Ac I_\EE^\trans D_V \tilde Y^\trans \\
&= - Y I_\EE \Ac D_{k,\EE}(x) I_\EE^\trans \tilde Y^\trans 
\end{align*}
with $D_{k,\EE}(x) \in \pd$.
Hence, we have
\begin{align*}
- \tilde z^\trans D' B^\trans \! \JJ(x) D B \, \tilde z
&= \tilde z^\trans D' B^\trans Y I_\EE \Ac D_{k,\EE}(x) I_\EE^\trans \tilde Y^\trans D B \, \tilde z \\
&= z^\trans B^\trans (Y I_\EE) \Acc (\tilde Y I_\EE)^\trans D B \, \tilde z \\
&= a^\trans \Acc \, \tilde a ,
\end{align*}
\blue{where} we introduced $z = D' \tilde z$ (in the same open orthant as $\tilde z$) 
and further
\begin{align*}
a &= (Y I_\EE)^\trans B \, z , \\
\tilde a &= (\tilde Y I_\EE)^\trans D B \, \tilde z .
\end{align*}

As stated above, for every nonzero $\tilde z \in \R^n$,
we will find $D' \in \pd$
such that $- \tilde z^\trans D' B^\trans \! \JJ(x) D B \, \tilde z > 0$.
Ultimately,
we will find $z=D'\tilde z$ and a chain graph $(V,\EE)$
such that $a^\trans \Acc \, \tilde a > 0$.

Let $O \subset \R^s$ be an open orthant and $\tilde z \in O$ be nonzero.
By assumption, $\sign(S) \cap \sign(\tilde S^\perp) = \{0\}$, that is, $\Sigma(S) \cap \tilde S^\perp = \{0\}$.
Recall $\ker ( \tilde Y I_\Omega)^\trans = (\im (\tilde Y I_\Omega) )^\perp = \tilde S^\perp$
and $D B \, \tilde z \in \Sigma(S)$.
Hence, $\tilde a = (\tilde Y I_\Omega)^\trans D B \, \tilde z \neq 0$
and further 
$
\tilde \tau = \sign \left( \tilde a \right) \neq 0
$.
Obviously, $\tilde \tau \in \sign \big( (\tilde Y I_\Omega)^\trans (\Sigma(B(O))) \big) =: \mathcal{\tilde T}$.

Analogously,
let $\mathcal{T} := \sign \big( (Y I_\Omega)^\trans (B(O)) \big)$
and $\mathcal{T'} := \sign \big( (Y I_\Omega)^\trans (B(\overline{O})) \big)$,
where $\overline{O}$ is the corresponding closed orthant.
Clearly, $\mathcal{T} \subseteq \mathcal{T'}$.

By assumption, $ \mathcal{\tilde T} \subseteq \mathcal{T'}^\updownarrow$.
The remaining case distinction (to identify $z=D'\tilde z$ and a chain graph $(V,\EE)$) is exactly as in the proof of Proposition~\ref{pro:Pmatrix}.
\end{proof}

As a final result, we can extend Theorem~\ref{thm:n} to the general setting.

\begin{thm} \label{thm:n_S}
Let $(G_k,y,\tilde y)$ be a generalized mass-action system,
where $S = \im B$ with an orthonormal basis $B$.
If the sign conditions in Proposition~\ref{pro:S} hold
and $B^\trans \! J(x) B$ is sign-symmetric, 
then, for all rate constants,
every positive CBE is linearly stable (in its stoichiometric class)
\end{thm}

\begin{rem}
Again, the conclusions hold for all generalized mass-action systems
(with the given stoichiometric and kinetic-order complexes)
that are based on a graph 
with the same (partition of the vertex set into) components,
but arbitrary edge set.
\end{rem}

\subsubsection*{Strategy II}
\blue{Alternatively,} one considers the original Jacobian matrix
and generalizes standard stability results.

For $A \in \R^{n \times n}$ and $r=\rank  A < n$,
we extend the definition of certain matrix classes.
Recall that $A$ is a $P$-matrix ($P_0$-matrix) if all its principal minors are positive (nonnegative),
and $A$ is a $P_0^+$-matrix if all its principal minors are nonnegative
and at least one of every order is positive.
Now, if $r<n$,
we call $A$ a $P_0^+$-matrix (of order $r$) if all its principal minors are nonnegative
and at least one of every order up to $r$ is positive.

\begin{pro} \label{pro:nec}
Let $A \in \R^{n \times n}$ and $L = \im A$.
If $A$ is $D$-stable on $L$, then $-A$ is a $P_0^+$-matrix (of order $r=\rank A$).
\end{pro}
\begin{proof}
For $D \in \pd$, the characteristic polynomial of $AD$ is given by
\[
\det(\la \id- AD) = \la^{n-r} (\la^r + c_{1} \la^{r-1} + \ldots + c_{r}) ,
\]
where
\[
c_{k} = \sum_{i_1, i_2, \ldots, i_k} M^k_{i_1, i_2, \ldots, i_k} \, d_{i_1} d_{i_2} \cdots d_{i_k} ,
\]
$M^k_{i_1, i_2, \ldots, i_k}$ are the principal minors of order $k$ of $-A$, 
and $d_i$ are the (positive) diagonal entries of $D$.

If $AD$ is stable on $L$, then it has 
$r$ eigenvalues with negative real part
and $n-r$ zero eigenvalues.
By the Routh-Hurwitz Theorem,
$c_k>0$ for all $k=1,\ldots,r$.
If $A$ is $D$-stable on $L$,
this holds for all~$d_i$.
Hence, for all $k=1,\ldots,r$, we have
$M^k_{i_1, i_2, \ldots, i_k} \ge 0$ for all tuples $(i_1, i_2, \ldots, i_k)$
and $M^k_{i_1, i_2, \ldots, i_k} > 0$ for some $(i_1, i_2, \ldots, i_k)$.
\end{proof}

\begin{thm} \label{thm:ballantine}
Let $A \in \R^{n \times n}$ have positive leading principal minors up to order $r=\rank A$.
Then, there is a diagonal matrix $D \in \pd$ such that $DA$ has 
$r$ simple eigenvalues with positive real part
and $n-r$ zero eigenvalues.
\end{thm}
\begin{proof}
Let $A_1 \in \R^{r \times r}$ be the leading principal submatrix of $A$ of order $r$
(having positive leading principal minors).
By \cite[Theorem~1]{FisherFuller1958} or \cite[Theorem~1]{Ballantine1970},
there is a diagonal matrix $D_1 \in \pd$ such that $D_1A_1$ has
$r$ positive, simple eigenvalues.
Let $A$ and a corresponding diagonal matrix $D \in \pd$ have the partitions
\[
A = \begin{pmatrix} A_1 & A_2 \\ A_3 & A_4 \end{pmatrix} , \,
D = \begin{pmatrix} D_1 & 0 \\ 0 & d \, \id \end{pmatrix} ,
\]
where $A_4, \id \in \R^{(n-r) \times (n-r)}$, $\id$ is the identity, and $d>0$.
Now, let
\[
M(d) = D A = \begin{pmatrix} D_1 A_1 & D_1 A_2 \\ d \, A_3 & d \, A_4 \end{pmatrix} .
\]
In particular,
$M(0)$ has $r$ positive, simple eigenvalues (of $D_1 A_1$)
and $n-r$ zero eigenvalues.
For $d>0$ small enough,
$M(d)$ has $r$ simple eigenvalues with positive real part
(perturbations of the eigenvalues of $D_1 A_2$)
and $n-r$ zero eigenvalues (corresponding to $\ker A$).
\end{proof}

In terms of stability (on $L=\im A$), Theorem~\ref{thm:ballantine} takes the following form.
\begin{thm} \label{thm:ballantine2}
Let $A \in \R^{n \times n}$.
If $-A$ has positive leading principal minors up to order $r=\rank A$,
then there is a diagonal matrix $D \in \pd$ such that $DA$ is stable on $L=\im A$.
\end{thm}

The additional assumption of sign-symmetry 
(and the nonnegativity of principal minors of order equal to the rank)
ensures stability.
\begin{thm} \label{thm:carlson}
Let $A \in \R^{n \times n}$ be sign-symmetric.
Further, let $-A$ have positive leading principal minors up to order $r=\rank A$
and nonnegative principal minors of order $r$.
Then, $A$ is $D$-stable on $L=\im A$.
\end{thm}
\begin{proof}
First, we show that $A$ is stable on $L$.

Since $A$ is sign-symmetric, it has no nonzero eigenvalues on the imaginary axes,
by~\cite[Lemma~2.2]{HershkowitzKeller2003}.
Since $-A$ has positive leading principal minors up to order~$r$,
there is a diagonal matrix $D^* \in \pd$ such that $D^*A$ is stable on $L=\im A$,
by~Theorem~\ref{thm:ballantine2}.
Now, for $D \in \pd$,
the characteristic polynomial of $DA$ is given by
\[
\det(\la \id- DA) = \la^{n-r} (\la^r + c_{1} \la^{r-1} + \ldots + c_{r}) ,
\]
where
\[
c_{k} = \sum_{i_1, i_2, \ldots, i_k} M^k_{i_1, i_2, \ldots, i_k} \, d_{i_1} d_{i_2} \cdots d_{i_k} ,
\]
$M^k_{i_1, i_2, \ldots, i_k}$ are the principal minors of order $k$ of $-A$, 
and $d_i$ are the (positive) diagonal entries of $D$.
Since $-A$ has nonnegative principal minors of order~$r$ and $M^r_{1,2,\ldots,r}>0$,
we have $c_r>0$, and $DA$ has exactly $n-r$ zero eigenvalues. 

Finally, consider the homotopy $D_t = (1-t)D^*+t\,\id$ for $t\in[0,1]$ with $D_0= D^*$ and $D_1=\id$.
Since $D_t \in \pd$, every matrix $D_t A$ is sign-symmetric
and hence has no nonzero eigenvalues on the imaginary axes.
In fact, as just shown, it has $n-r$ zero eigenvalues.
By continuity, since $D^*A$ is stable on $L$, also $A$ is stable on $L$.
(Otherwise, some $D_t A$ has an additional zero eigenvalue or a conjugate pair of eigenvalues on the imaginary axis.)

Since, for every $D \in \pd$, $AD$ is sign-symmetric, and minors of $A$ and $AD$ have the same signs,
also $AD$ is stable on $L$.
\end{proof}

\blue{
Ultimately, we use Theorem~\ref{thm:carlson} to ensure the $D$-stability of the scaled Jacobian matrix $\JJ$
and hence the stability of the original Jacobian matrix $J$.

To summarize, 
in case of a proper stoichiometric subspace ${S \subset \R^n}$,
the linear stability of every CBE (in its stoichiometric class) can be guaranteed using two alternative strategies:
either (I) by showing that the negative of the {\em reduced} Jacobian matrix is a $P$-matrix (and sign symmetric),
see Theorem~\ref{thm:n_S},
or (II) by showing that the negative of the {\em scaled} Jacobian matrix has a P-matrix of order $s=\dim S$ as a principal submatrix
and nonnegative principal minors of order $s$ (and is sign symmetric), see Theorem~\ref{thm:carlson}.
}




Finally, we apply Theorem~\ref{thm:carlson} to two examples from the biochemical literature
\blue{and discuss the applicability of Theorem~\ref{thm:n_S}.}
%

\begin{exa} We study a two-component signaling system, 
consisting of a histidine kinase $\ce{X}$ that auto-phosphorylates and transfers the phosphate group to a response regulator $\ce{Y}$ 
that undergoes auto-dephosphorylation, cf.~\cite{ConradiFeliuMinchevaWiuf2017}.

Traditionally, the dynamical system is specified as a network with MAK:
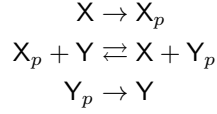
\begin{align*}
\ce{X} &\to \ce{X}_p \\ \ce{X}_p + \ce{Y} & \rightleftarrows \ce{X} + \ce{Y}_p \\ \ce{Y}_p & \to \ce{Y}
\end{align*}
The network has 6 complexes, is not weakly reversible, and has deficiency $\delta=6-3-2=1$.

However, it can equivalently be specified as a network with GMAK:
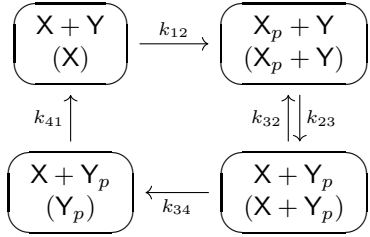
\[
\begin{tikzcd}[ampersand replacement=\&]
\mbox{\ovalbox{$\begin{array}{c} \ce{X} + \ce{Y} \\ (\ce{X}) \end{array}$}} \arrow[r,"k_{12}"] \& \mbox{\ovalbox{$\begin{array}{c} \ce{X}_p + \ce{Y} \\ (\ce{X}_p+\ce{Y}) \end{array}$}} \arrow[d,xshift=+0.5ex,"k_{23}"] \\
\mbox{\ovalbox{$\begin{array}{c} \ce{X} + \ce{Y}_p \\ (\ce{Y}_p) \end{array}$}} \arrow[u,"k_{41}"] \& \mbox{\ovalbox{$\begin{array}{c} \ce{X} + \ce{Y}_p \\ (\ce{X}+\ce{Y}_p)\end{array}$}} \arrow[u,xshift=-0.5ex,"k_{32}"] \arrow[l,"k_{34}"]
\end{tikzcd}
\]
This network has 4 \blue{vertices}, 
is weakly reversible, and has $\delta = 4-1-2=1$, 
$\tilde \delta = 4-1-3=0$ (since $\dim S=2$, $\dim \tilde S = 3$).
Hence, there exist CBE for all rate constants.
\blue{Further, the network has effective deficiency $\delta'=3-1-2=0$
(4~vertices, but only 3~different stoichiometric complexes).
Hence, there are no other steady states.
Finally,}
note that reaction $3 \to 4$ does not contribute to the dynamics, but makes the network weakly reversible.

Let the chemical species be ordered as $(\ce{X},\ce{X}_p,\ce{Y},\ce{Y}_p)$. Then, 
\[
Y = 
\begin{pmatrix}
 1 & 0 & 1 & 1 \\
 0 & 1 & 0 & 0 \\
 1 & 1 & 0 & 0 \\
 0 & 0 & 1 & 1
\end{pmatrix} ,
\quad 
\tilde Y =
\begin{pmatrix}
 1 & 0 & 1 & 0 \\
 0 & 1 & 0 & 0 \\
 0 & 1 & 0 & 0 \\
 0 & 0 & 1 & 1
\end{pmatrix} ,
\quad \text{and} \quad
S = \im \begin{pmatrix} -1 & 0 \\ 1 & 0 \\ 0 & -1 \\ 0 & 1 \end{pmatrix} .
\]
%
Further, let $\la_\x, \la_\xx>0$ be the parameters of the full cycle
$C_\x = \{1 \to 2 \to 3 \to 4 \to 1\}$ and the two-cycle $C_\xx = \{2 \rightleftarrows 3\}$, respectively, and let
\[
A_\x = 
\begin{pmatrix}
 -1 & 0 & 0 & 1 \\
 1 & -1 & 0 & 0 \\
 0 & 1 & -1 & 0 \\
 0 & 0 & 1 & -1 
\end{pmatrix} ,
\quad
A_\xx = 
\begin{pmatrix}
 0 & 0 & 0 & 0 \\
 0 & -1 & 1 & 0 \\
 0 & 1 & -1 & 0 \\
 0 & 0 & 0 & 0 
\end{pmatrix} 
\]
be the corresponding Laplacian matrices.
Using cycle decomposition~\eqref{eq:cycle},
\begin{align*}
- \JJ &= - Y (\la_\x A_\x + \la_\xx A_\xx) \, \tilde Y^\trans \\
&= \begin{pmatrix}
 \la_\x+\la_\xx & -\la_\x-\la_\xx & -\la_\x-\la_\xx & \la_\xx \\
 -\la_\x-\la_\xx & \la_\x+\la_\xx & \la_\x+\la_\xx & -\la_\xx \\
 -\la_\xx & \la_\x+\la_\xx & \la_\x+\la_\xx & -\la_\x-\la_\xx \\
 \la_\xx & -\la_\x-\la_\xx & -\la_\x-\la_\xx & \la_\x+\la_\xx 
\end{pmatrix} .
\end{align*}
\normalsize
In particular, $\im \JJ = S$ and $-\JJ$ is a $P_0^+$-matrix, both necessary conditions for stability.
As sufficient conditions (via Theorem~\ref{thm:carlson}),
we check that $-\JJ$ is sign-symmetric,
has a $P$-matrix of order $r=\rank \JJ = \dim S = 2$ as a principal submatrix,
and only nonnegative principal minors of order 2.
Indeed, using a computer algebra system, we find that all conditions are fulfilled.
For example, the principal submatrix corresponding to the species $(\ce{X},\ce{Y})$ is a $P$-matrix.

\newcommand{\sst}[1]{{\scriptstyle #1}}

\blue{As it turns out, 
also} the conditions of Theorem~\ref{thm:n_S} hold,
in particular, the sign conditions in Proposition~\ref{pro:S}.
\blue{
To show this, we use

\vspace{-2ex}
\footnotesize
\[
S = \im B
\quad \text{with} \quad
B = \begin{pmatrix} -1 & 0 \\ 1 & 0 \\ 0 & -1 \\ 0 & 1 \end{pmatrix} ,
\quad
\tilde S^\perp = \im \begin{pmatrix} 0 \\ -1 \\ 1  \\ 0 \end{pmatrix} ,
\]
\[
I_\Omega = \bordermatrix{ & \sst{12} & \sst{13} & \sst{14} & \sst{23} & \sst{24} & \sst{34} \cr 
\sst{1} & -1 & -1 & -1 & 0 & 0 & 0 \cr 
\sst{2} & 1 & 0 & 0 & -1 & -1 & 0 \cr 
\sst{3} & 0 & 1 & 0 & 1 & 0 & -1 \cr 
\sst{4} & 0 & 0 & 1 & 0 & 1 & 1
} , \quad \text{and}
\]
\[
Y I_\Omega = 
\begin{pmatrix} 
-1 & 0 & 0 & 1 & 1 & 0 \\
-1 & 0 & 0 & 1 & 1 & 0 \\
0 & -1 & -1 & -1 & -1 & 0 \\
0 & 1 & 1 & 1 & 1 & 0
\end{pmatrix}
, \quad
\tilde Y I_\Omega = 
\begin{pmatrix} 
-1 & 0 & 0 & 1 & 1 & 0 \\
-1 & 0 & 0 & 1 & 1 & 0 \\
0 & -1 & -1 & -1 & -1 & 0 \\
0 & 1 & 1 & 1 & 1 & 0
\end{pmatrix}
.
\]
\normalsize
Clearly, $\sign(S) \cap \sign(\tilde S^\perp)=\{0\}$ in Proposition~\ref{pro:S} is fulfilled,
and it remains to check

\vspace{-2ex}
\footnotesize
\[
\sign \left( (\tilde Y I_\Omega)^\trans (\Sigma( B(O) )) \right) 
\subseteq 
\sign \left( (Y I_\Omega)^\trans (B(\overline O)) \right)^\updownarrow , 
\quad \text{ for all open orthants } O \subset \R^2 . 
\]

\normalsize
For every orthant $O$, we determine the orthants $\Sigma( B(O) )$ that $B(O)$ intersects
and the sign vectors in $\sign ( \im (\tilde Y I_\Omega)^\trans )$ they generate (via linear programs).
For every such sign vector $\tilde \tau$,
we check if the set $B(\overline O)$ generates a ``corresponding'' sign vector $\tau$ (with $\tau\neq0$ and $\tilde\tau\le\tau$ or $\tilde\tau\ge\tau$)
in $\sign ( \im (Y I_\Omega)^\trans )$.
Indeed, the condition is fulfilled.

Recall that Theorem~\ref{thm:n_S} guarantees} 
the stability of {\em all} generalized mass-action systems with the same (stoichiometric and kinetic-order) complexes, 
but different \blue{(weakly reversible)} graphs.
\end{exa}

%
%

\begin{exa} We study a futile cycle involving two enzymes, substrate, product, and two enzyme-complexes,
cf.~\cite{Angeli2008} for a \blue{global} stability result and \cite{MuellerFlammStadler2022} for the chemistry of futile cycles.

Traditionally, the dynamical system is specified as a network with MAK:
\begin{align*}
\ce{E}+\ce{S} \rightleftarrows \ce{E^*} \to \ce{E}+\ce{P} \\
\ce{F}+\ce{S} \leftarrow \ce{F^*} \rightleftarrows \ce{F}+\ce{P}
\end{align*}
The network has 6 complexes, is not weakly reversible, and has deficiency $\delta=6-2-3=1$.

However, it can equivalently be specified as a network with GMAK:
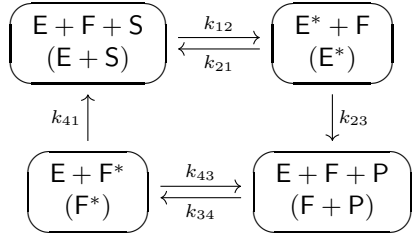
\[
\begin{tikzcd}[ampersand replacement=\&]
\mbox{\ovalbox{$\begin{array}{c} \ce{E}+\ce{F}+\ce{S} \\ (\ce{E}+\ce{S}) \end{array}$}} \arrow[r,yshift=+0.5ex,"k_{12}"] \& \mbox{\ovalbox{$\begin{array}{c} \ce{E^*} + \ce{F} \\ (\ce{E^*}) \end{array}$}} \arrow[l,yshift=-0.5ex,"k_{21}"] \arrow[d,xshift=+0ex,"k_{23}"] \\
\mbox{\ovalbox{$\begin{array}{c} \ce{E}+\ce{F^*} \\ (\ce{F^*}) \end{array}$}} \arrow[u,"k_{41}"] \arrow[r,yshift=+0.5ex,"k_{43}"] \& \mbox{\ovalbox{$\begin{array}{c} \ce{E}+\ce{F}+\ce{P} \\ (\ce{F}+\ce{P})\end{array}$}} \arrow[l,yshift=-0.5ex,"k_{34}"]
\end{tikzcd}
\]
This network has 4 complexes, is weakly reversible, and has $\delta = \tilde \delta = 4-1-3=0$
(since $\dim S=\dim \tilde S = 3$).
\blue{Hence, there exist CBE for all rate constants and no other steady states.}

Let the chemical species be ordered as $(\ce{E},\ce{E^*},\ce{F},\ce{F^*},\ce{S},\ce{P})$. Then,
\footnotesize
\[
Y = 
\begin{pmatrix}
 1 & 0 & 1 & 1 \\
 0 & 1 & 0 & 0 \\
 1 & 1 & 1 & 0 \\
 0 & 0 & 0 & 1 \\
 1 & 0 & 0 & 0 \\
 0 & 0 & 1 & 0 
\end{pmatrix} ,
\quad
\tilde Y =
\begin{pmatrix}
 1 & 0 & 0 & 0 \\
 0 & 1 & 0 & 0 \\
 0 & 0 & 1 & 0 \\
 0 & 0 & 0 & 1 \\
 1 & 0 & 0 & 0 \\
 0 & 0 & 1 & 0 
\end{pmatrix} ,
\quad \text{and} \quad
S = \im \begin{pmatrix} -1 & 0 & 0 \\ 1 & 0 & 0 \\ 0 & -1 & 0 \\ 0 & 1 & 0 \\ -1 & 0 & -1 \\ 0 & -1 & 1 \end{pmatrix} .
\]
\normalsize


Further, let $\la_\x, \la_\xxx, \la_\xxxx>0$ be the parameters of the full cycle
$C_\x = \{1 \to 2 \to 3 \to 4 \to 1\}$ and the two-cycles $C_\xxx = \{1 \rightleftarrows 2\}$ and $C_\xxxx = \{3 \rightleftarrows 4\}$, and let
\footnotesize
\[
A_\x = 
\begin{pmatrix}
 -1 & 0 & 0 & 1 \\
 1 & -1 & 0 & 0 \\
 0 & 1 & -1 & 0 \\
 0 & 0 & 1 & -1 
\end{pmatrix} ,
\quad
A_\xxx = 
\begin{pmatrix}
 -1 & 1 & 0 & 0 \\
 1 & -1 & 0 & 0 \\
 0 & 0 & 0 & 0 \\
 0 & 0 & 0 & 0 
\end{pmatrix} ,
\quad
A_\xxxx = 
\begin{pmatrix}
 0 & 0 & 0 & 0 \\
 0 & 0 & 0 & 0 \\
 0 & 0 & -1 & 1 \\
 0 & 0 & 1 & -1 
\end{pmatrix} 
\]
\normalsize
be the corresponding Laplacian matrices.
Using cycle decomposition~\eqref{eq:cycle},
\footnotesize
\begin{align*}
- \JJ &= - Y (\la_\x A_\x + \la_\xxx A_\xxx + \la_\xxxx A_\xxxx) \, \tilde Y^\trans \\
&= \begin{pmatrix}
 \la_\x+\la_\xxx & -\la_\x-\la_\xxx & 0 & 0 & \la_\x+\la_\xxx & 0 \\
 -\la_\x-\la_\xxx & \la_\x+\la_\xxx & 0 & 0 & -\la_\x-\la_\xxx & 0 \\
 0 & 0 & \la_\x+\la_\xxxx & -\la_\x-\la_\xxxx & 0 & \la_\x+\la_\xxxx \\
 0 & 0 & -\la_\x-\la_\xxxx & \la_\x+\la_\xxxx & 0 & -\la_\x-\la_\xxxx \\
 \la_\x+\la_\xxx & -\la_\xxx & 0 & -\la_\x & \la_\x+\la_\xxx & 0 \\
 0 & -\la_\x & \la_\x+\la_\xxxx & -\la_\xxxx & 0 & \la_\x+\la_\xxxx 
\end{pmatrix} .
\end{align*}
\normalsize
In particular, $\im \JJ = S$ and $-\JJ$ is a $P_0^+$-matrix, both necessary conditions for stability.
As sufficient conditions (via Theorem~\ref{thm:carlson}),
we check that $-\JJ$ is sign-symmetric,
has a $P$-matrix of order $r=\rank \JJ = \dim S = 3$ as a principal submatrix,
and only nonnegative principal minors of order 3.
Indeed, using a computer algebra system, we find that all conditions are fulfilled.
For example, the principal submatrix corresponding to the species $(\ce{E},\ce{F^*},\ce{P})$ is a $P$-matrix.

On the contrary, the conditions of Proposition~\ref{pro:S} do not hold.
They would guarantee the stability of {\em all} generalized mass-action systems with the same (stoichiometric and kinetic-order) complexes, 
but different \blue{(weakly reversible)} graphs.
However, for the ``reverse'' full cycle $C_\bullet = \{1 \to 4 \to 3 \to 2 \to 1\}$, 
the scaled Jacobian matrix $\JJ$ is not D-stable on $S$, 
since $-\JJ$ is not a $P_0^+$-matrix, cf.~Proposition~\ref{pro:nec}.
\end{exa}


\subsection{Epilogue on sign vector conditions} \label{sec:sign}

We have the following implications between the sign vector conditions in Theorem~\ref{thm:noother} and Proposition~\ref{pro:S}.

\vspace{-2ex}
\footnotesize
\begin{align*}
\sign(\tilde T) \subseteq \sign (T)^\updownarrow \quad & \quad \forall O \subset \R^s \colon 
\sign \left( (\tilde Y I_\Omega)^\trans (\Sigma( B(O) )) \right) 
\subseteq 
\sign \left( (Y I_\Omega)^\trans (B(\overline O)) \right)^\updownarrow 
\\
\rotatebox[origin=c]{45}{$\Downarrow$} \quad & \quad \rotatebox[origin=c]{-45}{$\Downarrow$} \\
\sign \left( (\tilde Y I_\Omega)^\trans (\Sigma(S)) \right) & \subseteq \sign (T)^\updownarrow 
\end{align*}
\normalsize
\blue{Here,} $s=\dim S$, and $O$ denotes an open orthant of $\R^s$.

The implication on the right has an important consequence.
Under the assumptions of Proposition~\ref{pro:S}, 
there are no other steady states in stoichiometric classes with a CBE, by Theorem~\ref{thm:noother}.

Under additional assumptions (on $S$ and $\tilde S$), the implication on the left becomes an equivalence.
\begin{lem}[\cite{MuellerRegensburger2012}, Corollary~3.8] \label{lem:dim}
Let $S,\tilde S$ be linear subspaces of $\R^n$.
If $\dim S=\dim \tilde S$,
then
\[
\sign(S) \cap \sign(\tilde S^\perp) = \{0\}
\quad \iff \quad
\sign(S^\perp) \cap \sign(\tilde S) = \{0\} .
\]
\end{lem}


\begin{lem} \label{lem:imA}
Let $S \subseteq \R^n$ be a linear subspace and $A \in \R^{a \times n}$ be a matrix.
Then,
\[
S^\perp \cap \im A^\trans = \{0\} 
\quad \implies \quad
A (\Sigma(S)) = \im A .
\]
\end{lem}
\begin{proof}
Let $S = \im B$ with $B \in \R^{n \times b}$ (and hence $S^\perp = \ker B^\trans$),
and assume $\ker B^\trans \cap \im A^\trans = \{0\} $.
Then,
\begin{align*}
A (\Sigma(S))
&= A (\Sigma(\im B)) \\
&= A ( \textstyle \bigcup_{D \in \pd} D(\im B) ) \\
&= \textstyle \bigcup_{D \in \pd} \im (A D B) \\
&= \textstyle \bigcup_{D \in \pd} ( \ker (B^\trans D A^\trans) )^\perp \\
&= ( \ker (B^\trans \id\, A^\trans) )^\perp \cup \textstyle \bigcup_{D \in \pd \setminus \{\id\}} \ldots \\
&= ( \ker A^\trans )^\perp \cup \ldots \\
&= \im A .
\end{align*}
\end{proof}

\begin{pro}
Let $(G_k,y,\tilde y)$ be a generalized mass-action system,
\[
\sign(S) \cap \sign(\tilde S^\perp) = \{0\} ,
\quad \text{and} \quad
\dim S=\dim\tilde S .
\]
Then,
$(\tilde Y I_\Omega)^\trans(\Sigma(S)) = \tilde T$.
\end{pro}
\begin{proof}
By Lemma~\ref{lem:dim} and Lemma~\ref{lem:imA} 
for $A = (\tilde Y I_\Omega)^\trans$ and $\im A^\trans = \tilde S$,
\begin{align*}
\sign(S) \cap \sign(\tilde S^\perp) = \{0\}
& \iff 
\sign(S^\perp) \cap \sign(\tilde S) = \{0\} \\
& \,\implies 
S^\perp \cap \tilde S = \{0\} \\
& \,\implies 
(\tilde Y I_\Omega)^\trans(\Sigma(S)) = \im (\tilde Y I_\Omega)^\trans = \tilde T .
\end{align*}
\end{proof}




\subsection*{Acknowledgments}


SM was supported by the Austrian Science Fund (FWF), grant P33218.
GR was supported by the Austrian Science Fund (FWF), grant P32301.
We thank two anonymous reviewers for their careful reading and helpful comments.



\bibliographystyle{abbrv} 
\bibliography{MR,linearstability,stability}

\begin{thebibliography}{10}

\bibitem{AlRadhawi2020}
M.~A. Al-Radhawi, D.~Angeli, and E.~D. Sontag.
\newblock A computational framework for a lyapunov-enabled analysis of
  biochemical reaction networks.
\newblock {\em {PLOS} Computational Biology}, 16(2):e1007681, Feb. 2020.

\bibitem{AlRadhawiAngeli2016}
M.~Ali Al-Radhawi and D.~Angeli.
\newblock New approach to the stability of chemical reaction networks:
  piecewise linear in rates {L}yapunov functions.
\newblock {\em IEEE Trans. Automat. Control}, 61(1):76--89, 2016.

\bibitem{Anderson2011}
D.~F. Anderson.
\newblock A proof of the global attractor conjecture in the single linkage
  class case.
\newblock {\em {SIAM} Journal on Applied Mathematics}, 71(4):1487--1508, Jan.
  2011.

\bibitem{Angeli2008}
D.~Angeli and E.~D. Sontag.
\newblock Translation-invariant monotone systems, and a global convergence
  result for enzymatic futile cycles.
\newblock {\em Nonlinear Analysis: Real World Applications}, 9(1):128--140,
  Feb. 2008.

\bibitem{BachemKern1992}
A.~Bachem and W.~Kern.
\newblock {\em Linear programming duality}.
\newblock Springer-Verlag, Berlin, 1992.

\bibitem{Ballantine1970}
C.~S. Ballantine.
\newblock Stabilization by a diagonal matrix.
\newblock {\em Proceedings of the American Mathematical Society},
  25(4):728--734, 1970.

\bibitem{Banaji2007}
M.~Banaji, P.~Donnell, and S.~Baigent.
\newblock $p$-matrix properties, injectivity, and stability in chemical
  reaction systems.
\newblock {\em {SIAM} Journal on Applied Mathematics}, 67(6):1523--1547, Jan.
  2007.

\bibitem{BanajiPantea2016}
M.~Banaji and C.~Pantea.
\newblock Some results on injectivity and multistationarity in chemical
  reaction networks.
\newblock {\em SIAM J. Appl. Dyn. Syst.}, 15:807--869, 2016.

\bibitem{BjornerLasSturmfelsWhiteZiegler1999}
A.~Bj{\"o}rner, M.~Las~Vergnas, B.~Sturmfels, N.~White, and G.~M. Ziegler.
\newblock {\em Oriented matroids}, volume~46 of {\em Encyclopedia Math. Appl.}
\newblock Cambridge University Press, Cambridge, second edition, 1999.

\bibitem{BlanchiniGiordano2014}
F.~Blanchini and G.~Giordano.
\newblock Piecewise-linear {L}yapunov functions for structural stability of
  biochemical networks.
\newblock {\em Automatica J. IFAC}, 50(10):2482--2493, 2014.

\bibitem{BorosHofbauer2019}
B.~Boros and J.~Hofbauer.
\newblock Planar {S}-systems: Permanence.
\newblock {\em J. Differential Equations}, 266(6):3787--3817, 2019.

\bibitem{BorosHofbauerMueller2017}
B.~Boros, J.~Hofbauer, and S.~M{\"u}ller.
\newblock On global stability of the {L}otka reactions with generalized
  mass-action kinetics.
\newblock {\em Acta Appl. Math.}, 151(1):53--80, 2017.

\bibitem{BorosHofbauerMuellerRegensburger2018}
B.~Boros, J.~Hofbauer, S.~M{\"u}ller, and G.~Regensburger.
\newblock The center problem for the {L}otka reactions with generalized
  mass-action kinetics.
\newblock {\em Qual. Theory Dyn. Syst.}, 17(2):403--410, 2018.

\bibitem{BorosHofbauerMuellerRegensburger2019}
B.~Boros, J.~Hofbauer, S.~M{\"u}ller, and G.~Regensburger.
\newblock Planar {S}-systems: Global stability and the center problem.
\newblock {\em Discrete Contin. Dyn. Syst. Ser. A}, 39(2):707--727, 2019.

\bibitem{BorosMuellerRegensburger2020}
B.~Boros, S.~M{\"u}ller, and G.~Regensburger.
\newblock Complex-balanced equilibria of generalized mass-action systems:
  Necessary conditions for linear stability.
\newblock {\em Mathematical Biosciences and Engineering}, 17(1):442--459, 2020.

\bibitem{Carlson1974}
D.~Carlson.
\newblock A class of positive stable matrices.
\newblock {\em Journal of Research of the National Bureau of Standards, Section
  B: Mathematical Sciences}, 78B(1):1, Jan. 1974.

\bibitem{ConradiFeliuMinchevaWiuf2017}
C.~Conradi, E.~Feliu, M.~Mincheva, and C.~Wiuf.
\newblock Identifying parameter regions for multistationarity.
\newblock {\em {PLOS} Computational Biology}, 13(10):e1005751, Oct. 2017.

\bibitem{CraciunDickensteinShiuSturmfels2009}
G.~Craciun, A.~Dickenstein, A.~Shiu, and B.~Sturmfels.
\newblock Toric dynamical systems.
\newblock {\em J. Symbolic Comput.}, 44(11):1551--1565, 2009.

\bibitem{CraciunMuellerPanteaYu2019}
G.~{Craciun}, S.~{M{\"u}ller}, C.~{Pantea}, and P.~{Yu}.
\newblock {A generalization of Birch's theorem and vertex-balanced steady
  states for generalized mass-action systems}.
\newblock {\em Mathematical Biosciences and Engineering}, 16(6):8243--8267,
  2019.

\bibitem{Cross1978}
G.~W. Cross.
\newblock Three types of matrix stability.
\newblock {\em Linear Algebra Appl.}, 20(3):253--263, 1978.

\bibitem{Feinberg1972}
M.~Feinberg.
\newblock Complex balancing in general kinetic systems.
\newblock {\em Arch. Rational Mech. Anal.}, 49:187--194, 1972/73.

\bibitem{Feinberg2019}
M.~Feinberg.
\newblock {\em Foundations of chemical reaction network theory}, volume 202 of
  {\em Applied Mathematical Sciences}.
\newblock Springer, Cham, 2019.

\bibitem{FeliuMuellerRegensburger2019}
E.~Feliu, S.~M{\"u}ller, and G.~Regensburger.
\newblock Characterizing injectivity of classes of maps via classes of
  matrices.
\newblock {\em Linear Algebra and its Applications}, 580:236--261, 2019.

\bibitem{FiedlerPtak1962}
M.~Fiedler and V.~Pt{\'{a}}k.
\newblock On matrices with non-positive off-diagonal elements and positive
  principal minors.
\newblock {\em Czechoslovak Mathematical Journal}, 12(3):382--400, 1962.

\bibitem{FisherFuller1958}
M.~E. Fisher and A.~T. Fuller.
\newblock On the stabilization of matrices and the convergence of linear
  iterative processes.
\newblock {\em Mathematical Proceedings of the Cambridge Philosophical
  Society}, 54(4):417--425, Oct. 1958.

\bibitem{HershkowitzKeller2003}
D.~Hershkowitz and N.~Keller.
\newblock Positivity of principal minors, sign symmetry and stability.
\newblock {\em Linear Algebra and its Applications}, 364:105--124, May 2003.

\bibitem{Horn1972}
F.~Horn.
\newblock Necessary and sufficient conditions for complex balancing in chemical
  kinetics.
\newblock {\em Arch. Rational Mech. Anal.}, 49:172--186, 1972/73.

\bibitem{HornJackson1972}
F.~Horn and R.~Jackson.
\newblock General mass action kinetics.
\newblock {\em Arch. Rational Mech. Anal.}, 47:81--116, 1972.

\bibitem{Johnston2014}
M.~D. Johnston.
\newblock Translated chemical reaction networks.
\newblock {\em Bull. Math. Biol.}, 76:1081--1116, 2014.

\bibitem{Johnston2015}
M.~D. Johnston.
\newblock A computational approach to steady state correspondence of regular
  and generalized mass action systems.
\newblock {\em Bull. Math. Biol.}, 77:1065--1100, 2015.

\bibitem{JohnstonMuellerPantea2019}
M.~D. Johnston, S.~M{\"u}ller, and C.~Pantea.
\newblock A deficiency-based approach to parametrizing positive equilibria of
  biochemical reaction systems.
\newblock {\em Bull. Math. Biol.}, 81(4):1143--1172, 2019.

\bibitem{Kandori1993}
M.~Kandori, G.~J. Mailath, and R.~Rob.
\newblock Learning, mutation, and long run equilibria in games.
\newblock {\em Econometrica}, 61(1):29--56, 1993.

\bibitem{lotka:1910}
A.~J. Lotka.
\newblock Contribution to the theory of periodic reactions.
\newblock {\em J. Phys. Chem.}, 14(3):271--274, 1910.

\bibitem{lotka:1920:a}
A.~J. Lotka.
\newblock Analytical note on certain rhythmic relations in organic systems.
\newblock {\em Proc. Natl. Acad. Sci.}, 6(7):410--415, 1920.

\bibitem{lotka:1920:b}
A.~J. Lotka.
\newblock Undamped oscillations derived from the law of mass action.
\newblock {\em J. Am. Chem. Soc.}, 42:1595--1599, 1920.

\bibitem{Minty1974}
G.~J. {Minty}.
\newblock {A ``from scratch'' proof of a theorem of Rockafellar and Fulkerson.}
\newblock {\em {Math. Program.}}, 7:368--375, 1974.

\bibitem{Mueller2022}
S.~M{\"u}ller.
\newblock On a new decomposition of the graph {L}aplacian and the binomial
  structure of mass-action systems.
\newblock {\em Journal of Nonlinear Science}, 33(91), 2023.

\bibitem{Mueller2016}
S.~M{\"u}ller, E.~Feliu, G.~Regensburger, C.~Conradi, A.~Shiu, and
  A.~Dickenstein.
\newblock Sign conditions for injectivity of generalized polynomial maps with
  applications to chemical reaction networks and real algebraic geometry.
\newblock {\em Found. Comput. Math.}, 16(1):69--97, 2016.

\bibitem{MuellerFlammStadler2022}
S.~M{\"u}ller, C.~Flamm, and P.~F. Stadler.
\newblock What makes a reaction network "chemical"?
\newblock {\em Journal of Cheminformatics}, 14:63, 2022.

\bibitem{MuellerHofbauerRegensburger2019}
S.~M{\"u}ller, J.~Hofbauer, and G.~Regensburger.
\newblock On the bijectivity of families of exponential/generalized polynomial
  maps.
\newblock {\em SIAM J. Appl. Algebra Geom.}, 3(3):412--438, 2019.

\bibitem{MuellerRegensburger2012}
S.~M\"uller and G.~Regensburger.
\newblock Generalized mass action systems: {C}omplex balancing equilibria and
  sign vectors of the stoichiometric and kinetic-order subspaces.
\newblock {\em SIAM J. Appl. Math.}, 72(6):1926--1947, 2012.

\bibitem{MuellerRegensburger2014}
S.~M\"uller and G.~Regensburger.
\newblock Generalized mass-action systems and positive solutions of polynomial
  equations with real and symbolic exponents.
\newblock In V.~P. Gerdt, W.~Koepf, E.~W. Mayr, and E.~H. Vorozhtsov, editors,
  {\em Computer Algebra in Scientific Computing. Proceedings of the 16th
  International Workshop (CASC 2014)}, volume 8660 of {\em Lecture Notes in
  Comput. Sci.}, pages 302--323, Berlin/Heidelberg, 2014. Springer.

\bibitem{Rockafellar1970}
R.~T. Rockafellar.
\newblock {\em Convex analysis}.
\newblock Princeton University Press, Princeton, N.J., 1970.

\bibitem{SiegelJohnston2008}
D.~Siegel and M.~Johnston.
\newblock Linearization of complex balanced chemical reaction systems.
\newblock {\em Unpublished}, 2008.
\newblock
  \href{https://www.researchgate.net/publication/253758553_Linearization_of_Complex_Balanced_Chemical_Reaction_Systems}{Linearization\_of\_Complex\_Balanced\_Chemical\_Reaction\_Systems}.

\bibitem{SiegelJohnston2011}
D.~Siegel and M.~D. Johnston.
\newblock A stratum approach to global stability of complex balanced systems.
\newblock {\em Dyn. Syst.}, 26(2):125--146, 2011.

\bibitem{TonelloJohnston2018}
E.~Tonello and M.~D. Johnston.
\newblock Network {T}ranslation and {S}teady-{S}tate {P}roperties of {C}hemical
  {R}eaction {S}ystems.
\newblock {\em Bull. Math. Biol.}, 80:2306--2337, 2018.

\bibitem{Tutte1948}
W.~T. Tutte.
\newblock The dissection of equilateral triangles into equilateral triangles.
\newblock {\em Proc. Cambridge Philos. Soc.}, 44:463--482, 1948.

\bibitem{Ziegler1995}
G.~M. Ziegler.
\newblock {\em Lectures on polytopes}.
\newblock Springer-Verlag, New York, 1995.

\end{thebibliography}


\newpage

\appendix 

\section*{Appendix} \label{app}


In the main text,
we introduce generalized mass-action systems in Section~\ref{sec:gmas}
and present our main results on complex-balanced equilibria (CBE) in Section~\ref{sec:main}.

In this appendix,
we provide basic results regarding 
order theory and oriented matroids.
First,
we consider (total) preorders.
Further, 
we present fundamental results on sign vectors of linear subspaces,
using a general theorem of the alternative.

\section{(Total) preorders} \label{app:order}

A {\em preorder} $\le$ on a set $V$ is a binary relation on $V$ such that 
$i \le i$ (reflexivity) and $i \le j$ and $j \le k$ imply $i \le k$ (transitivity)
for all $i,j,k \in V$.
We write $i \sim j$ if $i \le j$ and $j \le i$ (equivalence)
and $i < j$ if $i \le j$ and $j \not\le i$ (strictness).
%

A preorder $\le$ on $V$ is {\em total} if $i \le j$ or $j \le i$ for all $i,j \in V$.
If $\le$ is total, then $i < j$ is equivalent to $j \not\le i$ for all $i,j \in V$.
%

Two preorders $\le_1,\le_2$ on $V$ are {\em harmonious} (a symmetric relation)
if $i <_1 j$ implies $j \not<_2 i$ (and vice versa, that is, $i <_2 j$ implies $i \not<_1 j$) for all $i,j \in V$.
Two total preorders $\le_1,\le_2$ on $V$ are harmonious if and only if $i <_1 j$ implies $i \le_2 j$
(and vice versa, that is, $i <_2 j$ implies $i \le_1 j$) for all $i,j \in V$.

\begin{lem} \label{lem:order}
Let $\le_1,\le_2$ be two total preorders on $V$ that are harmonious.
Then there exists a total strict order $<$ on $V$ such that $i < j$ implies $i \le_1 j$ and $i \le_2 j$.
\end{lem}
\begin{proof}
For $i \in V$, let $V_i = \{ j \in V \mid i \sim_1 j \text{ and } i \sim_2 j \}$
and choose some total strict order $<$ on every equivalence class $V_i$. 
Now, let $i,j \in V$. 
If $V_i=V_j$ (that is, $i,j$ are equivalent),
assume $i < j$.
Then, $i \le_1 j$ and $i \le_2 j$, as claimed.
If $V_i\not=V_j$, assume $i <_1 j$ or $i <_2 j$ and set $i < j$. 
In any case, $i \le_1 j$ and $i \le_2 j$, as claimed.
\end{proof}

Given two preorders $\le_1,\le_2$ on $V$, $\le_1$ {\em conforms to} $\le_2$ (a non-symmetric relation)
if $\le_1,\le_2$ are harmonious and $i \sim_2 j$ implies $i \not<_1 j$ and $j \not<_1 i$.
Given two total preorders $\le_1,\le_2$ on $V$, $\le_1$ conforms to $\le_2$ 
if and only if $\le_1,\le_2$ are harmonious and $i \sim_2 j$ implies $i \sim_1 j$.


\section{Sign vectors of subspaces} \label{app:sign}

{\bf Additional notation.}
For a sign vector $\sigma \in \{-,0,+\}^n$, we introduce 
\[
\sigma^- = \{ i \mid \sigma_i = -\}, \quad \sigma^0 = \{ i \mid \sigma_i = 0\}, \quad \text{and} \quad \sigma^+ = \{ i \mid \sigma_i = +\} .
\]
Further, we write $\sigma \perp \tau$ ($\sigma$ and $\tau$ are orthogonal)
if either $\sigma_i \tau_i = 0$ for all $i$ or there exist $i,j$ with $\sigma_i \tau_i = -$ and $\sigma_j \tau_j = +$.
For $\mathcal{T} \subseteq \{-,0,+\}^n$, we introduce the orthogonal complement
\[
\mathcal{T}^\perp = \{ \sigma \in \{-,0,+\}^n \mid \sigma \perp \tau \text{ for all } \tau \in \mathcal{T} \} \, .
\]

As in \cite[Appendix B]{MuellerHofbauerRegensburger2019}, we recall a general theorem of the alternative for subspaces of $\R^n$ 
that allows to easily derive theorems of the alternative for sign vectors of a linear subspace and its orthogonal complement. 
For the relation to standard theorems of the alternative,
see \cite{Minty1974};
for the corresponding statements for arbitrary oriented matroids,
see \cite[Section 3.4]{BjornerLasSturmfelsWhiteZiegler1999} or~\cite[Chapter~5]{BachemKern1992}.

Let $x \in \R^n$, and let $I_1,\ldots,I_n$ be intervals of $\R$.
We define the interval
\begin{align*}
I(x) &\equiv x_1 I_1  + \ldots + x_n I_n \\ &= \{ x_1 y_1 + \ldots + x_n y_n \in \R \mid y_1 \in I_1, \ldots, y_n \in I_n \}
\end{align*}
and write $I(x) > 0$ if $y>0$ for all $y \in I(x)$.

\begin{thm}[``Minty's Lemma''. Theorem 22.6 in \cite{Rockafellar1970}] \label{minty}
Let $S$ be a linear subspace of~$\R^n$, and let $I_1,\ldots,I_n$ be intervals of $\R$.
Then one and only one of the following alternatives holds:
\begin{itemize}
\item[(a)]
There exists a vector $x = (x_1,\ldots,x_n)^\trans \in S$ such that
\[
x_1 \in I_1, \, \ldots,\, x_n \in I_n .
\]
\item[(b)]
There exists a vector $\xa = (\xa_1,\ldots, \xa_n)^\trans \in S^\perp$ such that
\[
\xa_1 I_1 + \ldots + \xa_n I_n > 0 .
\]
\end{itemize}
\end{thm}

We derive two consequences of Theorem~\ref{minty}.

\begin{cor} \label{minty_sign}
Let $S$ be a linear subspace of $\R^n$ and $\sigma \in \{ -,0,+ \}^n$ be a nonzero sign vector.
Then either (a) there exists a vector $x \in S$ with $x_i>0$ for $i \in \sigma^+$, $x_i<0$ for $i \in \sigma^-$, and $x_i=0$ for $i \in \sigma^0$
or (b) there exists a vector $\xa \in S^\perp$ with $\xa_i \ge 0$ for $i \in \sigma^+$, $\xa_i \le 0$ for $i \in \sigma^-$, 
and $\xa_i \neq 0$ for some $i \in \sigma^+ \cup \sigma^-$.

In terms of sign vectors, either $\sigma \in \sign(S)$
or there exists (a nonzero) ${\sigma^* \in \sign(S^\perp)}$ with $\sigma \hads \sigma^* > 0$.
\end{cor}
\begin{proof}
By Theorem~\ref{minty} with $I_i = (0,+\infty)$ for $i \in \sigma^+$, $I_i = (-\infty,0)$ for $i \in \sigma^-$, and $I_i = \{0\}$ otherwise.
\end{proof}

\begin{cor} \label{minty_other}
Let $S$ be a linear subspace of $\R^n$ and $\sigma \in \{ -,0,+ \}^n$ be a nonzero sign vector.
Then either (a) there exists a vector $x \in S$ with $x_i>0$ for $i \in \sigma^+$ and $x_i<0$ for $i \in \sigma^-$
or (b) there exists a nonzero vector $\xa \in S^\perp$ with $\xa_i \ge 0$ for $i \in \sigma^+$, $\xa_i \le 0$ for $i \in \sigma^-$, 
and $\xa_i = 0$ otherwise.

In terms of sign vectors, either there exists $\tau \in \sign(S)$ with $\tau \ge \sigma$
or there exists a nonzero $\tau^* \in \sign(S^\perp)$ with $\tau^* \le \sigma$.
\end{cor}
\begin{proof}
By Theorem~\ref{minty} with $I_i = (0,+\infty)$ for $i \in \sigma^+$, $I_i = (-\infty,0)$ for $i \in \sigma^-$, and $I_i = (-\infty,\infty)$ otherwise.
\end{proof}

\begin{pro} \label{minty_orthogonal}
Let $S$ be a linear subspace of $\R^n$. Then,
\[
\sign(S^\perp)=\sign(S)^\perp. 
\]
\end{pro}
\begin{proof}
($\subseteq$): Let $\sigma \in \sign(S^\perp)$ and $\tau \in \sign(S)$.
Now, let $u \in S^\perp$ and $v \in S$ such that $\sigma = \sign(u)$ and $\tau = \sign(v)$.
Then, $u \cdot v = 0$ implies $\sigma \perp \tau = 0$,
and hence $\sigma \in \sign(S)^\perp$.

($\supseteq$): Let $\sigma \notin \sign(S^\perp)$.
By Corollary~\ref{minty_sign}, there exists $\tau \in \sign(S)$ with $\sigma \hads \tau > 0$.
In particular, $\sigma \not\perp \tau$,
and hence $\sigma \notin \sign(S)^\perp$.
\end{proof}
This completes the proof of~\cite[
Corollary~53]{MuellerHofbauerRegensburger2019},
where we did not single out Corollary~\ref{minty_sign}.
For an alternative proof, using Farkas Lemma, see~\cite[Proposition~6.8]{Ziegler1995}.

\subsubsection*{Two subspaces}

The following new result allows an alternative formulation of the sign vector condition $\sign(S) \cap \sign(\tilde S^\perp) = \{0\}$,
which characterizes the uniqueness of CBE, cf.~Theorem~\ref{thm:unique}.

\begin{pro} \label{pro:intersect}
Let $S_1,S_2$ be linear subspaces of $\R^n$. 
Then, the following two statements are equivalent.
\begin{itemize}
\item[(i)]
$\sign(S_1) \cap \sign(S_2^\perp) = \{0\}$.
\item[(ii)]
For every nonzero $\sigma_1 \in \sign(S_1)$, there exists (a nonzero) $\sigma_2 \in \sign(S_2)$ with $\sigma \hads \tilde \sigma > 0$.
\end{itemize}
\end{pro}
\begin{proof}
$\neg$(ii) $\Rightarrow$ $\neg$(i): Assume that, for a nonzero $\sigma_1 \in \sign(S_1)$, 
there does not exist $\sigma_2 \in \sign(S_2)$ with $\sigma_1 \hads \sigma_2 > 0$.
By Corollary~\ref{minty_sign}, $\sigma_1 \in \sign(S_2^\perp)$,
and hence $\sigma_1 \in \sign(S_1) \cap \sign(S_2^\perp)$.

$\neg$(i) $\Rightarrow$ $\neg$(ii): Assume that there exists a nonzero $\sigma_1 \in \sign(S_1) \cap \sign(S_2^\perp)$.
By Corollary~\ref{minty_sign}, there does not exist $\sigma_2 \in \sign(S_2)$ with $\sigma_1 \hads \sigma_2 > 0$.
\end{proof}

Let $\mathcal{T} \subseteq \{-,0,+\}^n$. We define its {\em lower closure}, {\em upper closure}, and {\em total closure}, 
\begin{align*}
\mathcal{T}^\downarrow &= \{ \sigma \in \{-,0,+\}^n \mid \sigma \le \tau \text{ for some } \tau \in \mathcal{T} \} , \\
\mathcal{T}^\uparrow &= \{ \sigma \in \{-,0,+\}^n \mid \sigma \ge \tau \text{ for some nonzero } \tau \in \mathcal{T} \} , \\
\mathcal{T}^\updownarrow &= \{ \sigma \in \{-,0,+\}^n \mid \sigma \le \tau \text{ or } \sigma \ge \tau \text{ for some nonzero } \tau \in \mathcal{T} \} .
\end{align*}
Clearly,  
$\mathcal{T}^\updownarrow = \mathcal{T}^\downarrow \cup \mathcal{T}^\uparrow$.

By definition, we have the following fact, related to Proposition~\ref{pro:intersect}.
\begin{fac} \label{fac:closure}
Let $S_1,S_2$ be subsets of $\R^n$. 
Then, the following two statements are equivalent.
\begin{itemize}
\item[(i)]
$\sign(S_1) \subseteq \sign(S_2)^\updownarrow$.
\item[(ii)]
For every nonzero $\sigma_1 \in \sign(S_1)$, there exists a nonzero $\sigma_2 \in \sign(S_2)$ with $\sigma_1 \le \sigma_2$ or $\sigma_1 \ge \sigma_2$.
\end{itemize}
\end{fac}

The following lemma is implied by Proposition~\ref{pro:intersect} and Fact~\ref{fac:closure}.
We also provide a direct proof.

\begin{lem} \label{cc_i}
Let $S_1,S_2$ be linear subspaces of $\R^n$. \\
If $\sign(S_1) \subseteq \sign(S_2)^\updownarrow$, 
then $\sign(S_1) \cap \sign(S_2^\perp) = \{0\}$.
\end{lem}
\begin{proof}
Assume that there exists a nonzero $\sigma_1 \in \sign(S_1) \cap \sign(S_2^\perp)$.
If $\sign(S_1) \subseteq \sign(S_2)^\updownarrow$,
then there exists a nonzero $\sigma_2 \in \sign(S_2)$ with $\sigma_1 \le \sigma_2$ or $\sigma_1 \ge \sigma_2$.
In any case, $\sigma_1 \not\perp \sigma_2$,
thereby contradicting
$\sigma_1 \in \sign(S_2^\perp) = \sign(S_2)^\perp$, by~Proposition~\ref{minty_orthogonal}, and $\sigma_2 \in \sign(S_2)$.
\end{proof}

Obviously, we have the implications
\begin{equation}
\begin{gathered}
\sign(S_1) = \sign(S_2) \\ 
\Downarrow \\
\sign(S_1) \subseteq \sign(S_2) \\
\rotatebox[origin=c]{-45}{$\Downarrow$} \qquad \rotatebox[origin=c]{45}{$\Downarrow$} \\
\sign(S_1) \subseteq \sign(S_2)^\downarrow
\qquad
\sign(S_1) \subseteq \sign(S_2)^\uparrow \\
\rotatebox[origin=c]{45}{$\Downarrow$} \qquad \rotatebox[origin=c]{-45}{$\Downarrow$} \\
\sign(S_1) \subseteq \sign(S_2)^\updownarrow \\
\Downarrow \\
\sign(S_1) \cap \sign(S_2^\perp) = \{0\} .
\end{gathered}
\end{equation}


\end{document}